\documentclass[reqno,11pt,a4paper]{article}
\usepackage{verbatim}   % useful for program listings
\usepackage{color}      % use if color is used in text
\usepackage{hyperref}   % use for hypertext links, including those to external documents and URLs
\usepackage{amsmath}
 \usepackage{algorithm,algorithmic,refcount}
 \usepackage{setspace}
 \usepackage{bbm}

\usepackage{mathtools}
\usepackage{amsfonts}
\usepackage{amssymb}
\usepackage{booktabs}
\usepackage{multirow}
\usepackage[pdftex,dvipsnames]{xcolor}  % Coloured text etc.
\usepackage{xargs}                      % Use more than one optional parameter in a new commands
\usepackage{cleveref}
\usepackage{graphicx}
\usepackage{epstopdf,epsfig,subfigure}
\usepackage{curve2e}
\usepackage{hyperref}   % use for hypertext links, including those to external documents and URLs

\usepackage{array}

\providecommand{\keywords}[1]
{
  \small	
  \textbf{\textit{Keywords---}} #1
}

\providecommand{\MSC}[1]
{
  \small	
  \textbf{\textbf{Mathematics Subject Classification }} #1
}

\usepackage[colorinlistoftodos,prependcaption,textsize=tiny]{todonotes}
\usepackage[body={16cm,22cm},centering]{geometry}
\usepackage{amsthm,amsfonts,mathrsfs,graphicx}
\usepackage[font=small,labelfont=bf]{caption}
\newtheorem{definition}{Definition}[section]
\newtheorem{theorem}{Theorem}[section]
\newtheorem{lemma}{Lemma}[section]
\newtheorem{remark}{Remark}[section]

{\theoremstyle{definition}
\newtheorem{example}{Example}[section]}

\newtheorem{assumption}{Assumption}[section]

\newcommand{\eqnum}{\refstepcounter{equation}\textup{\tagform@{\theequation}}}

\newcommandx{\info}[2][1=]{\todo[linecolor=OliveGreen,backgroundcolor=OliveGreen!25,bordercolor=OliveGreen,#1]{#2}}

\providecommand{\abs}[1]{\lvert#1\rvert}

\DeclareMathOperator*{\essinf}{ess\,inf}
\DeclareMathOperator*{\esssup}{ess\,sup}

\DeclareMathOperator{\spn}{span}

\title{Analysis of RHC for stabilization of nonautonomous parabolic equations under uncertainty}
\author{Behzad Azmi\thanks{Department of Mathematics and Statistics, University of Konstanz, D-78457 Konstanz, Germany.  E-mail: behzad.azmi@uni-konstanz.de}\and
Lukas Herrmann\thanks{Johann Radon Institute for Computational and Applied Mathematics (RICAM), Austrian Academy of Sciences,  Altenbergerstra\ss e 69, A-4040 Linz, Austria. E-mail: lukas.herrmann@ricam.oeaw.ac.at}\and
Karl Kunisch\thanks{Johann Radon Institute for Computational and Applied Mathematics (RICAM), Austrian Academy of Sciences and Department of Mathematics,  University of Graz, Heinrichstra\ss e 36, A-8010 Graz,  Austria. E-mail:karl.kunisch@uni-graz.at}}
\date{}                     %% if you don't need date to appear

\begin{document}
\date{\today}
\maketitle

\begin{abstract}
 Stabilization of a class of time-varying parabolic equations with uncertain input data using Receding Horizon Control (RHC) is investigated.  The diffusion coefficient and the initial function are prescribed as random fields. We consider both cases, uniform and log-normal distributions of the diffusion coefficient.
 The controls are chosen to be finite dimensional  and enter into the system  as a linear combination of finitely many indicator functions (actuators)  supported in open subsets of the spatial domain.    Under suitable regularity assumptions, we study the expected (averaged) stabilizability of the RHC-controlled system with respect to the number of actuators.  An upper bound is also obtained for the failure probability of RHC in relation to the choice of the number of actuators and parameters in the equation.

\end{abstract}

\keywords{receding horizon control,   random evolution PDEs,  uncertainty,  
averaged stabilizability,  random fields,  nonautonomous parabolic equations}

\MSC{93C20 · 35R60 · 93D20 · 93E03}
%\begin{keywords}
%model predictive control,   random evolution PDEs ,uncertainty,   feedback stability
%\end{keywords}

\section{Introduction}
Mathematical models arising  in real-world applications  are typically affected by uncertainties that can lead to significant differences between the real systems response and the corresponding deterministic mathematical models. Therefore it is of great interest  for applications to include uncertainty in these  models and quantify its effect on the predicted quantities of interest.  Such uncertainty may reflect our ignorance or inability to properly characterize all input parameters of the mathematical
model, and it may also describe an intrinsic variability of the physical system, see e.g., \cite{chiba_stochastic_2009,chiba_stochastic_2008}.  Probability theory provides a natural framework to describe and deal with such uncertainties which are characterized  as random variables or more generally random fields.

 We investigate stabilization of the controlled systems governed by  the following  linear parabolic equation utilizing  the receding horizon control (RHC) strategy
\begin{equation}
\tag{CS}
\label{CS}
\begin{cases}
\partial_t y- \nabla \cdot ( \nu(\omega)\nabla y)+ a(t)y+  \nabla \cdot (b(t)y) = \sum^{N}_{i =1} u_i(t) \mathbf{1}_{O_i}   &   (t,x) \in (0,\infty)\times D,\\
y =0   & (t,x)  \in  (0,\infty) \times \partial D,\\
y(0)=y_0(\omega) &  x \in D,
\end{cases}
\end{equation}
where $D\subset \mathbb{R}^n$ is a bounded rectangular with  boundary $\partial D$  and the  functions $\mathbf{1}_{O_i}$,  represent the actuators. They are modelled as the characteristic functions related to open sets $O_i \subset D$ for  $i= 1, \dots, N$.   The reaction term $a(t) = a(t, x)$ and the convection term $b(t) = b(t, x)$ are, respectively,  real- and $\mathbb{R}^n$-valued functions of $(t, x) \in (0, \infty) \times D$.   Further,  the diffusion coefficient (the convective heat transfer coefficient)  $\nu(\omega) =\nu(\omega,x)$ and the initial function  $y_0(\omega) = y_0(x,\omega)$  are very difficult to measure in practice and,  hence,  they are affected by a certain amount of uncertainty  with  $x \in D$  and $\omega \in \Omega$.   These uncertainty inputs are described as random fields defined on the  complete probability space $(\Omega, \mathcal{F}, \mathbb{P})$,  where $\Omega$ denotes the set of outcomes,  $\mathcal{F}$ is the associated $\sigma$-algebra of events,  and $\mathbb{P} : \mathcal{F} \to [0, 1]$ is a probability measure.

%Here,
%complete means that $\mathcal{F}$  contains all $\mathbb{P}$-null sets,  that is if $A \subset %\Omega$ is such that there exists
%$B \in  \mathcal{F}$ with $ A \subset B$ and $\mathbb{P}(B) = 0$, then $A \in \mathcal{F}$.

% Controlling or stabilizing  the controlled system would require knowledge  of the actual value of the random parameter.  But   in practice this is impossible,  and thus, controls should be obtained independently  of the values of random fields. {\tcr {dieser satz ist mir unklar, wie ist er logisch mit dem übrigen text verlinkt ??. weglassen ?}}

 In this work,  we aim at  deriving a stabilizing control that is robust with respect to  the perturbations of the dynamics caused by all possible realizations of the random parameter $\omega$. For this purpose we consider the notion of averaged stability and verify  that the expected value of the distance of the state to the steady state with respect to the random parameter converges asymptotically to zero.  More concretely,  the control objective is to find a (spatially) finite-dimensional control  $ \mathbf{u} \in L^2((0,\infty);U)$   for which
\begin{equation*}
\mathbb{E} \left[ \|y(t)\|^2_{L^2(D; \mathbb{R})} \right] := \int_{\Omega} \|y(t,\omega )\|^2_{L^2(D; \mathbb{R})} d \mathbb{P}(\omega)   \to 0     \quad  \text{ as } \quad  t \to \infty 
\end{equation*}
holds.  Here we will consider the both cases deterministic $U=\mathbb{R}^N$ and stochastic $U= L^2_{\mathbb{P}}(\Omega;\mathbb{R})\otimes \mathbb{R}^N$ controls.  The stabilizing control $\mathbf{u}$ is computed by a receding horizon framework.  In this framework,  the current control action is obtained by minimizing a performance index  defined on a finite time interval,  ranging from the current time $t_0$  to some future time $t_0+T$, with $T \in (0, \infty]$ and $t_0\in(0,\infty)$.  Here we set
\begin{equation}
\label{e1}
J_{T}(\mathbf{u};t_0,y_0):= \frac{1}{2}\int_{t_0}^{t_0+T} \mathbb{E}\left[\ell(t,y(t))\right]\,dt +\frac{\beta}{2}\int^{t_0+T}_{t_0} \|\mathbf{u}(t)\|_U^2  dt,
\end{equation}
for $\beta\geq 0$,  $\ell:  \mathbb{R}_{\geq 0}  \times H^1_0(D;\mathbb{R}) \to \mathbb{R}_{\geq 0}$  satisfying $\ell(t,y)\geq \alpha_{\ell} \|y\|^2_{L^2(D;\mathbb{R})}$ with $\alpha_{\ell}>0$,  and $\ell(t,0) = 0$. As an example one may consider $\ell(t, y)=\alpha_{\ell} \|y\|^2_{L^2(D;\mathbb{R})}$.
 Then,  the stabilization of the control system \eqref{CS} can be formulated as the following infinite-horizon optimal control problem
\begin{equation}
\label{opinf} \tag*{$OP_{\infty}(y_0)$}
\min\{ J_{\infty}(\mathbf{u};0,y_0) \mid (y,\mathbf{u}) \text{ satisfies } \eqref{CS},  \mathbf{u} \in   L^2((0,\infty);U)\}.
\end{equation}
As we will show,  the receding horizon approach delivers a suboptimal approximation to the solution of \ref{opinf}.  This approximation  is constructed by concatenation of a sequence of finite horizon optimal controls defined on overlapping intervals covering $(0, \infty)$.  These finite horizon problems  have the following form.  For a given initial time $\bar{t}_0$,  initial functions $ \bar{y}_0 =  \bar{y}_0(\omega,x)$,  and prediction horizon $T$ consider
\begin{align}
\label{optT}
\tag*{$OP_{T}(\bar{t}_0,\bar{y}_0)$}
&\min_{\mathbf{u}\in   L^2((\bar{t}_0,\bar{t}_0+T);U)}J_T(\mathbf{u};  \bar{t}_0, \bar{y}_0)\\\label{e41}
\text{ s.t } &\begin{cases}
\partial_t y- \nabla \cdot ( \nu(\omega)\nabla y)+ a(t)y+ \nabla \cdot (b(t)y) = \sum^{N}_{i =1} u_i(t) \mathbf{1}_{O_i}   &   (t,x) \in (\bar{t}_0,\bar{t}_0+T)\times D,\\
y =0   & (t,x)  \in  (\bar{t}_0,\bar{t}_0+T) \times \partial D,\\
y(\bar{t}_0)=\bar{y}_0(\omega) &  x \in D.
\end{cases}
\end{align}
In the receding horizon framework,  we define sampling instances $t_k :=k\delta$,  for $k=0,1,2,\dots$,  and for a chosen sampling time $\delta >0$.  Then,  at every current sampling instance $t_k$   sampling time $\delta >0$.  Then,  at every current sampling instance $t_k$ with state $ y_{rh}(t_k) \in  L^2_{\mathbb{P}}(\Omega;\mathbb{R}) \otimes L^2(D;\mathbb{R})$,  an  open-loop optimal control problem $OP_{T}(t_k, y_{rh}(t_k))$ is solved over a finite prediction horizon $[t_k,t_k+T]$ for an appropriate prediction horizon $T>\delta$.  Then,  the associated optimal control is applied to steer the system from time $t_k$ with the initial state $ y_{rh}(t_k) \in  L^2_{\mathbb{P}}(\Omega;\mathbb{R}) \otimes L^2(D;\mathbb{R})$ until time $t_{k+1}:=t_k+\delta$ at which point,  a new measurement of the state    $ y_{rh}(t_{k+1}) \in  L^2_{\mathbb{P}}(\Omega;\mathbb{R}) \otimes L^2(D;\mathbb{R})$,  is assumed to be available.  The process is repeated starting from  this new measured state:  we obtain a new optimal control and a new predicted state trajectory by shifting the prediction horizon forward in time. The sampling time $\delta$ is the time period between two sample instances. Throughout, we denote the receding horizon state- and control variables  by $y_{rh}(\cdot,\cdot)$ and $\mathbf{u}_{rh}(\cdot)$, respectively. Also,  $(y_T^*(\cdot,\cdot ;\bar{t}_0, \bar{y}_0), \mathbf{u}^*_T(\cdot; \bar{t}_0, \bar{y}_0))$ stands for the optimal state and control of the optimal control problem with finite time horizon $T$,  and initial function  $\bar{y}_0$ at initial time $\bar{t}_0$. This is summarized in  Algorithm \ref{RRHA}.
\begin{algorithm}[htbp]
%\floatname{algorithm}{RHA}
\caption{Robust RHC($\delta,T$)}\label{RRHA}
%\begin{spacing}{1.1}
\begin{algorithmic}[1]
\REQUIRE{The sampling time $\delta$,  the prediction horizon $T\geq \delta$,   and the initial state $y_0$}
\ENSURE{The stability of RHC~$\mathbf{u}_{rh}$.}
\STATE Set~$(\bar{t}_0,\bar{y}_0):=(0, y_0)$  and  $y_{rh}(0) =y_0 $;
%\STATE Set $y_{rh}(0)=y_0$,\; $c_{rh}(0) =c_0$,\;
%and $\dot{c}_{rh}(0) =c^1_0 $;
%\WHILE{$t_0<\overline T$}
%\STATE Set the initial vector~$\mathcal{I}_{0}=(t_0,y_0, c_{0},c^1_0)$;
\STATE Find the solution~$(y_T^*(\cdot;\bar{t}_0, \bar{y}_0)), \mathbf{u}^*_T(\cdot;\bar{t}_0, \bar{y}_0))$
over the time horizon~$(\bar{t}_0,\bar{t}_0+T)$ by solving the  open-loop problem~\ref{optT};
\STATE For all $\tau \in [\bar{t}_0,\bar{t}_0+\delta)$,  set $y_{rh}(\tau)=y^*_T(\tau;\bar{t}_0,\bar{y}_0)$ and $\mathbf{u}_{rh}(\tau)=\mathbf{u}^*_T(\tau; \bar{t}_0, \bar{y}_0)$;
\STATE Find a measurement  $ y_{rh}(\bar{t}_0 +\delta;\bar{t}_0,\bar{y}_0)$ of the state at time $\bar{t}_0+ \delta$;
\STATE Update: $(\bar{t}_0,\bar{y}_0)  \leftarrow (\bar{t}_0 +\delta,  y_{rh}(\bar{t}_0 +\delta; \bar{t}_0,\bar{y}_0))$;
\STATE Go to Step  2;
%  $y_{0}=y^*_T(t_0;\mathcal{I}_{0})$, $c_{0}=c^*_T(t_0;\mathcal{I}_{0})$,
%  and~$c_{0}^1=\dot c^*_T(t_0;\mathcal{I}_{0})$;
%\ENDWHILE
\end{algorithmic}
%\end{spacing}
\end{algorithm}
%For  controllability of REPDEs, see e.g, \cite{ lazar_stability_2018,lazar_averaged_2014,lu_averaged_2016,zuazua_averaged_2014}.
%Although control theory and stabilization have been studied well for deterministic infinite-dimensional controlled systems, there is very few research on infinite-dimensional systems under uncertainty.  Among them,  we can mention \cite{MR4335837, MR3764133,MR3210131,MR3465808}  in the context of  controllability results and   \cite{MR3682973,MR3548188}  for optimal control problems.  In particular,  to the best of our knowledge,  RHC has been still not studied for the control systems with uncertainty inputs.  In this project,  we take a step in this direction and,   relying on theoretical results in \cite{AzmiKK},  we a

Concerning the literature,  there is a growing interest in partial differential equations (PDEs) that involve some uncertainty.  So far, there are only a few papers investigating parabolic PDEs with random coefficients.  Here we can quote e.g., \cite{MR4274843,MR4246090, MR3115458,MR2583868,MR3022204}.
Concerning control and stabilization, which are well-studied for deterministic infinite-dimensional systems, and stochastic systems with the stochastic terms appearing in an affine manner, 
 there is little research on infinite-dimensional systems under uncertainty.  Among them,  we can mention \cite{MR4335837, MR3909811,MR3465808}  in the context of controllability results and   \cite{guth2022, MR3064588,MR3888962,MR3548188}  for optimal control problems.  To  our knowledge,  RHC has not yet been studied for  control systems with uncertainty inputs.  In this project,  we take a step in this direction and, relying on theoretical results in \cite{MR4022734},  we investigate the performance and stability of the receding horizon framework for \cref{CS} with both uniformly bounded and log-normally distributed random diffusions $\nu$.  For each case,  separately,  this involves  investigating the well-posedness of the state,  the stabilizability of the controlled system by (spatially) finite-dimensional controls,  and deriving continuity-  and observability-type of inequalities.  We also provide an upper bound for the failure probability for the receding horizon  framework depending on the choice of diffusion parameter $\nu$,   reaction and convection terms $a$ and $b$,   and the number $N$ of actuators.

The rest of the paper is organized as follows.  We start the next section by introducing the notation which is  used throughout the paper.  In Section 3, we consider time-varying parabolic equations with uniformly bounded random diffusion.  Under appropriate  assumptions, we study the well-posedness of the state equation, the stabilizability of the controlled system, and the stability of the receding horizon framework.   At the end of that section, we derive an upper bound for the failure probability of RHC.  In the forth section, we discuss the analogous  questions  for the case of  log-normally distributed random diffusion.

\section{Notation and  preliminaries}
 Let  Banach spaces~$X$ and~$Y$ be given. We write $X\xhookrightarrow{} Y$ if the inclusion $X\subseteq Y$ is continuous.
The space of continuous linear mappings from~$X$ into~$Y$ is denoted by~$\mathcal{L}(X,Y)$.  We also
write~$\mathcal{L}(X)\coloneqq\mathcal{L}(X,X)$.
The continuous dual of~$X$ is denoted by~$X'\coloneqq\mathcal{L}(X,  \mathbb{R})$.
The adjoint of an operator $L\in\mathcal{L}(X,Y)$ will be denoted with  $L^*\in\mathcal{L}(Y',X')$.

Let a Hilbert space  $H$ endowed with scalar product $(\cdot,\cdot)_H$ be given. Then the orthogonal complement to a given subset $B \subset H$ is denoted  by $B^{\perp}:=\{h \in H : (h,s)_H = 0\quad  \forall   s\in B \}$.

For any  two closed subspaces $\mathcal{F}$ and $\mathcal{G}$ of  the Hilbert space  $H$ satisfying $H = \mathcal{F} \oplus  \mathcal{G}$,   we define by $P^{\mathcal{G}}_{\mathcal{F}} \in \mathcal{L}(H, \mathcal{F})$ the oblique projection in $H$ onto $\mathcal{F}$   along   $\mathcal{G}$.  That is,  for every $h \in H$ if  we consider the unique decomposition $h = h_{\mathcal{F}}+h_{\mathcal{G}}$  with $h_{\mathcal{F}} \in \mathcal{F}$ and $h_{\mathcal{G}} \in \mathcal{G}$,  we have  $P^{\mathcal{G}}_{\mathcal{F}} h := h_{\mathcal{F}}$.  Then,  clearly,   $P^{\mathcal{F}^{\perp}}_{\mathcal{F}}$  is the orthogonal projection in $H$  onto $\mathcal{F}$ which is denoted by $P_{\mathcal{F}}$.

 For given Hilbert spaces $H_1$ and $H_2$,  we use the notation  $H_1 \otimes H_2$ for the tensor product of $H_1$ with $H_2$.

 We also consider the spaces $V : = H^1_0(D; \mathbb{R})$,   $V' := H^{-1}(D;\mathbb{R})$,  and $H:=L^2(D; \mathbb{R})$ endowed with their usual norms.  Then, for every open interval $(s_1, s_2)\subset \mathbb{R}_{\geq 0}$, we can define
\[ W(s_1,s_2):=\{ v \in L^2((s_1,s_2); V) :  \partial_t v \in L^2((s_1,s_2); V') \}, \]
 endowed with the norm
  $ \| v\|_{W(s_1,s_2)} :=  \left( \|v\|^2_{L^2((s_1,s_2); V)} + \| \partial_t v\|^2_{L^2((s_1,s_2); V')}  \right)^{\frac{1}{2}}$.

For the probability space $(\Omega, \mathcal{F}, \mathbb{P})$,  a Banach space $X$,  and $p\in [ 1,  \infty ]$,  we denote by  $L^p_{\mathbb{P}} (\Omega;X)$ the Lebesgue-Bochner space,  composed of all strongly measurable function $v:\Omega \to X$  whose norm is defined by
\begin{equation*}
\|v \|_{L^p_{\mathbb{P}}(\Omega;X)} := \begin{cases}  \left(\int_{\Omega} \| v(\cdot, \omega)\|^p_X d\mathbb{P}(\omega)   \right)^{\frac{1}{p}} &      p<\infty, \\
\esssup_{\omega \in \Omega} \| v(\cdot, \omega)\|_X & p = \infty.
      \end{cases}
\end{equation*}
We also  assume that $L^2_{\mathbb{P}} (\Omega; \mathbb{R})$  is a separable Hilbert space.  For this assumption it suffices to assume that $(\Omega , \mathcal{F} , \mathbb{P})$ is separable  see e.g., \cite[Theorem II.10]{MR751959}.  Then,  if $p = 2$ and $X$ is a separable Hilbert space,   the Bochner space $L^2_{\mathbb{P}} (\Omega;X)$ is isomorphic to the tensor product space $L^2_{\mathbb{P}} (\Omega; \mathbb{R})\otimes X$,  that is $L^2_{\mathbb{P}} (\Omega;X) \cong   L^2_{\mathbb{P}} (\Omega; \mathbb{R} )\otimes X$,  see e.g., \cite[Theorem 4.13]{MR2759829}.   

For the sake of convenience,  we abbreviate  $V_{\mathbb{P}} : =  L^2_{\mathbb{P}}(\Omega;H^1_0(D;\mathbb{R}))$,   $V_{\mathbb{P}}' := L^2_{\mathbb{P}}(\Omega;H^{-1}(D;\mathbb{R}))$, $H_{\mathbb{P}}:=L^2_{\mathbb{P}}(\Omega; L^2(D;\mathbb{R}))$ and
$U^{N}_{\mathbb{P}}:=L^2_{\mathbb{P}}(\Omega; \mathbb{R}^N)$.  By identifying  $H_{\mathbb{P}}$ with its dual 
we obtain a Gelfand triple $V_{\mathbb{P}}\hookrightarrow H_{\mathbb{P}} \hookrightarrow V'_{\mathbb{P}} $ of separable Hilbert spaces with dense injections.  Finally,   for every open interval $(s_1, s_2)\subset \mathbb{R}_{\geq 0}$, we consider the space $W_{\mathbb{P}}(s,t)$ defined
\[ W_{\mathbb{P}}(s_1,s_2):=\{ v \in L^2((s_1,s_2); V_{\mathbb{P}}) :  \partial_t v \in L^2((s_1,s_2); V'_{\mathbb{P}}) \}, \]
and endowed with the norm $ \| v\|_{W_{\mathbb{P}}(s_1,s_2)} :=  \left( \|v\|^2_{L^2((s_1,s_2); V_{\mathbb{P}})} + \| \partial_t v\|^2_{L^2((s_1,s_2); V'_{\mathbb{P}})}  \right)^{\frac{1}{2}}$.  It is well known
that $W_{\mathbb{P}}(s_1,s_2) \hookrightarrow C([s_1,s_2]; H_{\mathbb{P}})$.  Further, due to the fact that $L_{\mathbb{P}}^2(\Omega;\mathbb{R})$  is  separable,  we can write for any Hilbert  space $X$ that
\begin{equation*}
\begin{split}
L^2_{\mathbb{P}}(\Omega;\mathbb{R}) & \otimes  L^2((s_1,s_2);X)  \cong L^2_{\mathbb{P}}(\Omega;L^2((s_1,s_2);X)) \cong   L^2(\Omega\times(s_1,s_2) ; X)\\
& \cong L^2((s_1,s_2) ; L^2_{\mathbb{P}}( \Omega; X)) \cong  L^2((s_1,s_2);\mathbb{R}) \otimes  L^2_{\mathbb{P}}( \Omega; X).
\end{split}
\end{equation*}
Hence,  we can conclude
\begin{equation}
\label{e48}
L^2_{\mathbb{P}}(\Omega;\mathbb{R})  \otimes  W(s_1,s_2) \cong  W_{\mathbb{P}}(s_1,s_2).
\end{equation}
In the following,  we define the finite- and infinite-horizon value functions.  These will be used frequently in the analysis of RHC.
\begin{definition}
For any $y_0 \in  H_{\mathbb{P}}$   the infinite-horizon value function $V_{\infty}: H_{\mathbb{P}}\to \mathbb{R}_{\geq 0}$ is defined by
\begin{equation*}
V_{\infty}(y_0):= \min_{ \mathbf{u} \in L^2((0,\infty);U)}\{J_{\infty}(\mathbf{u};0,y_0) \text{ subject to  \eqref{CS}} \}.
\end{equation*}
Similarly, for every $(T, \bar{t}_0, \bar{y}_0) \in \mathbb{R}^2_{\geq 0} \times H_{\mathbb{P}}$,  the finite-horizon value function $V_{T}: \mathbb{R}_{\geq 0} \times H_{\mathbb{P}} \to \mathbb{R}_{\geq 0}$ is defined by
\begin{equation*}
V_{T}(\bar{t}_0, \bar{y}_0):= \min_{\mathbf{u} \in L^2(( \bar{t}_0, \bar{t}_0+T);U)}\{J_{T}(\mathbf{u};\bar{t}_0, \bar{y}_0 ) \text{ subject to \eqref{e41}} \}.
\end{equation*}
\end{definition}

\section{Parabolic PDEs with uniform random diffusion}
In this section we are concerned with the case when the diffusion coefficient is uniformly bounded away  from zero and from above.  This allows us to use  the  weak  formulation  directly. Throughout this section,   we impose the following conditions:
\begin{assumption}\label{assump}
%\phantom{ww}
We assume that:
\begin{itemize}
\item[A1:] There are random variables $\nu_{\min}$,  $\nu_{\max}$,  and  constants $\overline{\nu} $,  $\underline{\nu}$  such that
\begin{equation}
\label{e3}
0<\underline{\nu}  \leq \nu_{\min}(\omega)   \leq \nu(\omega,x) \leq  \nu_{\max}(\omega)  \leq   \overline{\nu} < \infty  \quad  \text{ for a.e.  }  x \in D  \text{ and }  \omega \in \Omega \text{ a.s.}.
\end{equation}
%\item[A2:] $y_0  \in L^2_{\mathbb{P}}(\Omega; L^2(D))$.
%\item[A3:] $u  \in   L^2((0,\infty);L^2_{\mathbb{P}}(\Omega, \mathbb{R}^N))$.
\item[A2:] For the reaction parameter $a$  and convection vector $b$,  we impose
 \begin{equation}
 \label{e56}
 \tag{RA}
 a \in L^{\infty}((0, \infty);  L^r(D;\mathbb{R})) \text{ with } r\geq n :=dim(D), \text{ and } b \in L^{\infty}((0,\infty)\times D ; \mathbb{R}^n),
 \end{equation}
   and set $\mathcal{N}(a,b):= \|a\|_{L^{\infty}((0,\infty);L^r(D;\mathbb{R}))}+\| b\|_{L^{\infty}((0,\infty)\times D ; \mathbb{R}^n)}$.
 \end{itemize}
\end{assumption}
We mention the two following examples for the diffusion $\nu$ satisfying A1.
\begin{example}
\label{Examp1}
 The case of the (truncated) log-normal fields,  i.e.,
\begin{equation}
\label{log-N}
 \nu(\omega,x) = \nu_{0}(x)+\exp(\sum^M_{j= 1}z_j(\omega) \psi_j(x)),
\end{equation}
where  $\psi_j \in L^{\infty}(D;\mathbb{R})$ for $j = 0,\dots,M$ and $\nu_0 \in L^{\infty}(D;\mathbb{R})$ with  $\essinf_{x \in D} \nu_0(x) = \underline{\nu}> 0$ .   The random variables $z_j: \Omega \to \mathbb{R}$ have zero means, they are pairwise uncorrelated,  and  they are  truncated  at some large enough lower and upper bounds,  see e.g.,  \cite[page 25]{martinez-frutos_optimal_2018} for more details.  For every $z  = (z_1,\dots,z_M) $,  the following quantities are assumed to be well defined.
\begin{equation}
\begin{split}
\nu_{\max}(\omega)  &= \esssup_{x\in D}\nu_{0}(x)+\exp(\sum^M_{j= 1}\abs{z_j(\omega)}\| \psi_j\|_{L^{\infty}(D;\mathbb{R})}), \\
\nu_{\min}(\omega)  & =  \essinf_{x \in D}\nu_{0}(x)+\exp( - \sum^M_{j= 1}\abs{z_j (\omega)}\| \psi_j\|_{L^{\infty}(D;\mathbb{R})}).
\end{split}
\end{equation}
Since the ranges of $z_j$ for $j = 1, \dots,M$ are bounded,  we have  \eqref{e3}  for  numbers  $\infty > \overline{\nu}\geq\underline{\nu}>0$.
\end{example}
\begin{example}
\label{Examp2}
 We can also consider the coefficient  $\nu$ to be characterized by a sequence of scalar random variables $\{ z_j  \}_{j \geq 1}$ with
\begin{equation}
\label{e8}
\nu(\omega,x)  = \nu_0(x) + \sum^{\infty}_{j= 1}z_j(\omega)\psi_j(x),
\end{equation}
where $\psi_j \in L^{\infty}(D;\mathbb{R})$ for $j\geq 1$,  and $z_j : \Omega \to \mathbb{R}$ for  $j = 1,2,\cdots$  are independent random variables which are distributed identically and uniformly in $[-1, 1]$ such that the range of  $z_j$ is in $[-1, 1]$.  Then all realizations of the random vector $z= (z_1,z_2,\dots)$ are supported in the cube $ [-1, 1]^{\mathbb{N}}$.    Further,  with $\nu^* := \essinf_{x \in D} \nu_{0}(x)$ and some $\kappa > 0$,  the functions $\psi_j$ are supposed to satisfy
\[
\sum^{\infty}_{j= 1} \| \psi_j  \|_{L^{\infty}(D;\mathbb{R})}  \leq  \frac{\kappa}{1+\kappa} \nu^* .
\]
This assumption implies that the fluctuations (resp., deviations) from mean of the random coefficient $\nu(x, \omega)$ in \eqref{e8} are dominated by the mean field, i.e., that they are small with respect to the deterministic mean field.  Then, we have
\begin{equation}
\begin{split}
\nu_{\max}(\omega)  &= \esssup_{x\in D} \nu_{0}(x)+\sum^{\infty}_{j= 1}\abs{ z_j(\omega)} \| \psi_j  \|_{L^{\infty}(D;\mathbb{R})}, \\
\nu_{\min}(\omega)  & =   \nu^*-\sum^{\infty}_{j= 1}\abs{ z_j(\omega)} \| \psi_j  \|_{L^{\infty}(D;\mathbb{R})},
\end{split}
\end{equation}
and the inequality in the right hand side of \eqref{e3} holds with $\underline{\nu}:= \nu^*-\frac{\kappa}{1+\kappa} \nu^* = \frac{1}{1+\kappa} \nu^*$.
\end{example}

\subsection{Well-posedness of  state equation}
We start with the well-posedness of \eqref{CS}.   In this regard,  we consider for $\omega \in \Omega$ a.s., the following auxiliary random linear parabolic equation
\begin{equation}
\label{e17}
\begin{cases}
\partial_t y- \nabla \cdot ( \nu(\omega,x)\nabla y)+ a(t,x)y+  \nabla \cdot (b(t,x)y) = f(t,x,\omega)  &   (t,x) \in (t_0,t_0+T)\times D,\\
y =0   & (t,x)  \in  (t_0,t_0+T) \times \partial D,\\
y(t_0)=y_0(\omega, x) &  x \in D,
\end{cases}
\end{equation}
and define the following notion of weak solution.
\begin{definition}
Let  $(T,t_0, y_0, f ) \in   \mathbb{R}^2 \times  H_{\mathbb{P}} \times L^2((t_0,t_0+T);V'_{_{\mathbb{P}}}) $ be given.  Then,  a random field  $y \in W_{\mathbb{P}}(t_0,t_0+T)$ is referred to as a weak solution of \eqref{e17},  if it satisfies
\begin{equation}
\label{e7}
\begin{split}
&\int^{t_0+T}_{t_0}\langle \partial_t y(t),\varphi(t)\rangle_{V'_{\mathbb{P}},V_{\mathbb{P}}}\,dt+  \int^{t_0+T}_{t_0} \int_{\omega}\int_D  \nu \nabla y(t)\cdot \nabla \varphi(t)  \,dx \, d \mathbb{P}(\omega) \,dt\\ &+ \int^{t_0+T}_{t_0}\langle a(t)y(t),\varphi(t) \rangle_{V'_{\mathbb{P}},V_{\mathbb{P}}}\,dt - \int^{t_0+T}_{t_0} \int_{\Omega}\int_D y(t)b(t) \cdot  \nabla \varphi  \,dx \, d \mathbb{P}(\omega) \,dt \\ &  =\int^{t_0+T}_{t_0} \langle f(t),\varphi \rangle_{V'_{\mathbb{P}},V_{\mathbb{P}}}\,dt \quad \text{ for all }\varphi \in L^2((t_0,t_0+T); V_{\mathbb{P}}),
\end{split}
\end{equation}
and $y(t_0)=y_0$ is satisfied in $H_{\mathbb{P}}$.  Here we use  $\langle \cdot, \cdot  \rangle_{V'_{\mathbb{P}},V_{\mathbb{P}}} :=  \mathbb{E} \left[  \langle \cdot, \cdot  \rangle_{V',V} \right]$.
\end{definition}
In the following we present the existence result and various a-priori estimates  for the weak solution of \eqref{e17}.   These estimates will be used frequently in the sequel.
\begin{theorem}
\label{Theo2}
For every multiple  $(T,t_0, y_0, f ) \in   \mathbb{R}_{\geq 0}^2 \times  H_{\mathbb{P}} \times L^2((t_0,t_0+T);V'_{\mathbb{P}})$ equation \eqref{e17} admits a unique weak random field $y \in W_{\mathbb{P}}(t_0,t_0+T)$ satisfying the following estimates
\begin{equation}
\label{e5}
\|y\|^2_{C([t_0,t_0+T];H_{\mathbb{P}})} + \| y \|^2_{W_{\mathbb{P}}(t_0,t_0+T)} \leq c_1\left( \|y_0\|^2_{H_{\mathbb{P}}} + \| f\|^2_{L^2((t_0,t_0+T);V'_{\mathbb{P}})} \right), %\\ \label{e13}
\end{equation}
with $c_1$ depending on  $(T,  \overline{\nu}, \underline{\nu},a,b,D)$.  Moreover, we have the following observability inequality
\begin{equation}
\label{e6}
\|y_0\|^2_{H_{\mathbb{P}}} \leq  c_{2} \left(1 + T^{-1} +  \mathcal{N}(a,b)\right)\|y\|^2_{L^2((t_0,t_0+T);V_{\mathbb{P}})}+ \|f\|^2_{L^2((t_0,t_0+T);V'_{\mathbb{P}})},
\end{equation}
with  $c_2$ depending only on $(T,  \overline{\nu}, \underline{\nu},D)$.
\end{theorem}
\begin{proof} From \eqref{e3} and Assumption A2 it follows that the sesquilinear form
\begin{equation}
\label{e4}
\begin{split}
 b( t, \psi ,\varphi ) &=   \int_{\Omega} \int_D  \nu \nabla  \psi \cdot  \nabla  \varphi  \,dx \, d \mathbb{P}(\omega) \,dt +\langle a(t)y(t),\varphi(t) \rangle_{V'_{\mathbb{P}},V_{\mathbb{P}}}\\& -  \int_{\Omega}\int_D y(t)b(t) \cdot  \nabla \varphi  \,dx \, d \mathbb{P}(\omega)    \qquad  \forall   \psi ,  \varphi  \in V_{\mathbb{P}},
 \end{split}
\end{equation}
is coercive and continuous. Thus  there exist positive constants $c_{\min}$, $c_{\max}$,  and $c_0$, such that for every $\psi,  \varphi \in V_{\mathbb{P}}$ and a.e.  $t \in (t_0,t_0+T)$ we have
\begin{equation}
 |b( t, \psi ,\varphi )|  \leq c_{\max} \| \psi \|_{V_{\mathbb{P}}}  \| \varphi \|_{V_{\mathbb{P}}},  \quad \text{ and } \quad
b( t, \psi ,\psi  ) \geq c_{\min} \|\psi \|_{V_{\mathbb{P}}}^2-c_0 \|\psi \|_{H_{\mathbb{P}}}^2.
\end{equation}
The rest of proof follows by  using a Galerkin approximation with   orthonormal basis functions  $\{ \psi_i \otimes \phi_j \}_{i,j \geq 1}  \subset L^2_{\mathbb{P}}(\Omega; \mathbb{R} ) \otimes V \cong V_{\mathbb{P}}$ and passing to the limit in the weak formations \eqref{e7}, where $\{ \psi_i \}_{i, \geq 1} $ and $\{  \phi_j \}_{j \geq 1}$ are orthonormal bases  for the spaces $L^2_{\mathbb{P}}(\Omega; \mathbb{R})$ and  $V$,  respectively.  The energy estimate \eqref{e5}  and \eqref{e6} are obtained  with the same arguments as in \cite[Proposition 3.2.]{MR4022734}.

%Finally, we come to the verification of the observability estimate \eqref{e14}.  Multiplying \eqref{e17} by $\frac{T+t_0-t}{T}y(t)$ and integrating in time from $t_0$ to $t_0+T$, we obtain
%\begin{multline}
%\label{e33}
%\int^{t_0+T}_{t_0}\frac{t_0+T-t}{T} \langle  \partial_t y(t),y(t)\rangle_{V',V}\,  dt =
%\int^{t_0+T}_{t_0}\frac{t_0+T-t}{T}  \Bigl( -\nu\|y(t)\|^2_{V} \\
% - \langle a(t)y(t),y(t) \rangle_{V',V} + (b(t)y(t), \nabla y(t))_H  +\langle f(t),y(t)\rangle_{V',V}  \Bigr)dt.
%\end{multline}
%By integration by part, we can infer that
%\begin{equation}
%\label{e34}
%\begin{split}
%\int^{t_0+T}_{t_0}\frac{t_0+T-t}{T} \langle  \partial_t y(t),y(t)\rangle_{V',V} \,dt&= \int^{t_0+T}_{t_0}  \frac{t_0+T-t}{2T} \left( \frac{d}{dt}\|y(t)\|^2_H \right) \,dt \\& = \frac{1}{2T}\int^{t_0+T}_{t_0}\|y(t)\|^2_H\,dt - \frac{1}{2}\| y(t_0)\|^2_H.
%\end{split}
%\end{equation}
%Now, using  \eqref{e18}, \eqref{e33}, \eqref{e34}, Young's inequality, and the fact that $\frac{T+t_0-t}{T} \leq 1$ for every  $t \in [t_0,t_0+T]$, we obtain
%\begin{multline}
%\|y_0 \|^2_H =\frac{1}{T} \int^{t_0+T}_{t_0}\| y(t) \|^2_H\,dt+ 2\int^{t_0+T}_{t_0}\frac{t_0+T-t}{T}\Bigl( \nu \|y(t)\|^2_{V} \\+ \langle a(t)y(t),y(t) \rangle_{V',V} - (b(t)y(t), \nabla y(t))_H  -\langle f(t),y(t)\rangle_{V',V}  \Bigr)dt
%\\ \leq c_p\left(\frac{1}{T}+\frac{2\nu+1}{c_p}+ c\,\mathcal{N}(a,b)\right)\int^{t_0+T}_{t_0}\| y(t)\|^2_V\,dt +\int^{t_0+T}_{t_0}\| f(t)\|^2_{V'}\,dt,
%\end{multline}
%where $c_p>0$ stands for the constant in the Poincar\'e  inequality. This implies \eqref{e14}.

\end{proof}

\subsection{Stabilizability  of the controlled system}
\label{Sec3}
In this section,  we study the stabilizability of \eqref{CS} with respect to the number of actuators.  Here we restrict ourselves to the rectangular domain,  that is $D = [0,L_d]^d \subset \mathbb{R}^d$ and  follow the arguments given in \cite[Theorem 4.1]{MR4242365}  and \cite{MR4207900}.
%We should mention that the proof of the stabilizability results is based on the decomposition of the space $H$ into two spaces.

 It is well-known that the Laplacian operator $-\Delta :  H^2(D; \mathbb{R}) \cap V \subset H   \to H$  has a compact inverse and,  thus,  there exists a nondecreasing  system of (repeated accordingly to their multiplicity) eigenvalues  $\{ \alpha_{i}  \}_{i \geq 1}$ with its associated complete basis satisfying
\begin{equation*}
 0< \alpha_1 \leq \alpha_2 \leq \cdots \leq \alpha_i \to \infty \text{ with } -\Delta e_i = \alpha_i e_i.
  \end{equation*}
For any $d\geq 1$,  we construct,  by induction,   a family of pairs   $ (\mathcal{O}_N, \mathcal{E}_N)$ such that $ H = \mathcal{O}_N\oplus \mathcal{E}^{\perp}_N $ for $N_{\sigma} = \sigma(N):=N^d$.

We start with  the one-dimensional case,  i.e.,  $d = 1$.  For this case  $D  =  (0,L_1)$  with  $L_1 > 0$,  and it has already been shown in
 \cite[Lems. 4.3 and 5.1]{RodSturm20} that $ L^2(D; \mathbb{R}) = \mathcal{O}_N\oplus \mathcal{E}^{\perp}_N $, if we take $N_{\sigma} = \sigma(N):=N$,  and for a fixed  $r\in (0,1)$ define the following sets
\begin{align*}
&\mathcal{E}_N=  \mathcal{E}^{[1]}_N := \spn \left\{ e^{[1]}_{i,N} : i \in \mathbf{N}\right \} \subset H^1_0((0,L_1); \mathbb{R})\\
&\mathcal{O}_N =\mathcal{O}^{[1]}_N := \spn \left \{ \mathbf{1}_{O^{[1]}_{i,N}} : i \in \mathbf{N} \right\}  \subset L^2((0,L_1); \mathbb{R}) \\
&O^{[1]}_{i,N} := (c^{[1]}_{i,N} -\frac{rL_1}{2N}, c^{[1]}_{i,N}+\frac{rL_1}{2N}), \quad c^{[1]}_{i,N} := \frac{(2i-1)L_1}{2N},  \quad  i\in \mathbf{N},
\end{align*}
 where  $\mathbf{N} := \{1,2,3,\dots,N\}$.  Further,   for  $i \in \mathbf{N}$,   $\mathbf{1}_{O^{[1]}_{i,N}}$  denote  the  indicator functions with supports  $O^{[1]}_{i,N}$  and   $e^{[1]}_{i,N}$  are the first eigenfunctions of the Laplacian in $L^2((0,L_1); \mathbb{R})$ under homogeneous  Dirichlet  boundary conditions.

Now,  we deal with higher-dimensional rectangular domains $D =  \bigtimes^d_{n = 1} (0,L_n)$.  Following  the results in  \cite[sect. 4.8.1]{KunRod19-cocv},   it can be shown that   the direct sum $L^2(D; \mathbb{R}) = \mathcal{E}_N \oplus \mathcal{O}^{\perp}_N $ property (note that  $  \mathcal{E}_N \oplus \mathcal{O}^{\perp}_N = \mathcal{O}_N \oplus \mathcal{E}^{\perp}_N $) holds also true for the choice $N_{\sigma} = \sigma(N):=N^d$ and the following setting
\begin{align*}
& \mathcal{E}_N := \spn \left\{ e^{\times}_{\mathbf{i},N}(x) = \bigtimes^{d}_{n =1} e^{[n]}_{i_n,S}(x_n)  : \mathbf{i}:=(i_1,\dots,i_d) \in \mathbf{N}^d \right\} \subset V \\
&\mathcal{O}_N := \spn \left\{ \mathbf{1}_{o^{\times}_{\mathbf{i},N}}(x)=\bigtimes^{d}_{n =1} \mathbf{1}_{O^{[n]}_{i_n,N}}(x_n)   : \mathbf{i} \in \mathbf{N}^d \right \} \subset H  \\
& O^{[n]}_{i,N} = (c^{[n]}_{i,N} -\frac{rL_n}{2N}, c^{[n]}_{i,N}+\frac{rL_n}{2N}), \quad c^{[n]}_{i,N} = \frac{(2i-1)L_n}{2N},  \quad  i\in \mathbf{N}, \\
&O^{\times}_{\mathbf{i},N}(x) = \bigtimes^{d}_{n=1} O^{[n]}_{\mathbf{i}_n,N}(x_n), \quad \mathbf{i} \in \mathbf{N}^d,
\end{align*}
where  $e^{[n]}_{i,N}$  with $i \in \mathbf{N}$ are the first eigenfunctions of the Laplacian  in  $L^2((0,L_n); \mathbb{R})$,   $\mathbf{i}:=(i_1,\dots,i_d)\in \mathbf{N}^d$,  and $x =(x_1,x_2,\dots,x_d) \in \bigtimes^{d}_{n=1}(0,L_n) $.  Moreover,  for this choice of the pair  $(\mathcal{O}_N, \mathcal{E}_N)$,  it can be proven,  with the same arguments as in \cite[Section 5]{MR4207900},  that  for every $N \in \mathbb{N}_0$ we have  $H = \mathcal{E}_N \oplus \mathcal{O}^{\perp}_N$,   and
\begin{equation}
\label{e12}
 \beta_N \geq  c_{\beta}N^2   \quad \text{with } \quad        \beta_N:= \inf_{Q\in( V\cap \mathcal{O}^{\perp}_N)  \backslash \{ 0\}}\frac{\| Q\|^2_V}{\|Q\|^2_H},
  \end{equation}
where  the constant  $c_{\beta}$ is  independent of $N$ and $Q$.  Figure \ref{fig.suppAc} illustrates the supports of actuators for the case $d = 2$ and different choices of $N$.
%%%%%%%%%%%%%%%%%%%%%%%%TIKZ BOXES

\setlength{\unitlength}{.0012\textwidth}
%Size

\newsavebox{\Rectfw}%
\savebox{\Rectfw}(0,0){%
\linethickness{3pt}
{\color{black}\polygon(0,0)(120,0)(120,80)(0,80)(0,0)}%
}%
\newsavebox{\Rectfg}%
\savebox{\Rectfg}(0,0){%
{\color{lightgray}\polygon*(0,0)(120,0)(120,80)(0,80)(0,0)}%
}%

\newsavebox{\Rectref}%
\savebox{\Rectref}(0,0){%
{\color{white}\polygon*(0,0)(120,0)(120,80)(0,80)(0,0)}%
%{\color{lightgray}\polygon*(30,20)(90,20)(90,60)(30,60)(30,20)}%
{\color{lightgray}\polygon*(45,30)(75,30)(75,50)(45,50)(45,30)}%
%\put(30,20){\scalebox{.5}{\usebox{\Rectfg}}}%
}%

%%%%%%%%%%%%%%%%%%%%%%%%%%

\begin{figure}[htbp!]
\begin{center}
\begin{picture}(500,100)%(0,0)

%\put(0,0){\GraphGrid(500,100)}%

%Rectref
\put(0,0){\usebox{\Rectfw}}%
\put(0,0){\usebox{\Rectref}}
% Rect2
 \put(190,0){\usebox{\Rectfw}}
 \put(190,0){\scalebox{.5}{\usebox{\Rectref}}}
 \put(250,0){\scalebox{.5}{\usebox{\Rectref}}}
 \put(250,40){\scalebox{.5}{\usebox{\Rectref}}}
 \put(190,40){\scalebox{.5}{\usebox{\Rectref}}}
% Rect3
 \put(380,0){\usebox{\Rectfw}}
 \put(380,0){\scalebox{.3333}{\usebox{\Rectref}}}
\put(420,0){\scalebox{.3333}{\usebox{\Rectref}}}
\put(460,0){\scalebox{.3333}{\usebox{\Rectref}}}
\put(380,26.6666){\scalebox{.3333}{\usebox{\Rectref}}}
\put(420,26.6666){\scalebox{.3333}{\usebox{\Rectref}}}
\put(460,26.6666){\scalebox{.3333}{\usebox{\Rectref}}}
\put(380,53.3333){\scalebox{.3333}{\usebox{\Rectref}}}
\put(420,53.3333){\scalebox{.3333}{\usebox{\Rectref}}}
\put(460,53.3333){\scalebox{.3333}{\usebox{\Rectref}}}
\put(40,85){$N=1$}
\put(230,85){$N=2$}
\put(420,85){$N=3$}
% \put(1,1){\blue $S=1$}
% \put(191,1){\blue $S=2$}
% \put(381,1){\blue $S=3$}

\end{picture}
\end{center}
% \caption{bbbbb.} \label{fig.bbbb}
% \end{figure}
%
%

%\vspace*{1em}
%
% \begin{figure}[h!]
\begin{center}
\begin{picture}(500,100)%(0,0)

%\put(0,0){\GraphGrid(500,100)}%
%
% %Rectref
%  \put(0,0){\usebox{\Rectfw}}
%  %
%  \put(0,0){\scalebox{.5}{\usebox{\Rectref}}}
%  \put(60,0){\scalebox{.5}{\usebox{\Rectref}}}
%  \put(60,40){\scalebox{.5}{\usebox{\Rectref}}}
%  \put(0,40){\scalebox{.5}{\usebox{\Rectref}}}
% % Rect4
 \put(0,0){\usebox{\Rectfw}}
 \put(0,0){\scalebox{.25}{\usebox{\Rectref}}}
\put(30,0){\scalebox{.25}{\usebox{\Rectref}}}
\put(60,0){\scalebox{.25}{\usebox{\Rectref}}}
\put(90,0){\scalebox{.25}{\usebox{\Rectref}}}
\put(0,20){\scalebox{.25}{\usebox{\Rectref}}}
\put(30,20){\scalebox{.25}{\usebox{\Rectref}}}
\put(60,20){\scalebox{.25}{\usebox{\Rectref}}}
\put(90,20){\scalebox{.25}{\usebox{\Rectref}}}
\put(0,40){\scalebox{.25}{\usebox{\Rectref}}}
\put(30,40){\scalebox{.25}{\usebox{\Rectref}}}
\put(60,40){\scalebox{.25}{\usebox{\Rectref}}}
\put(90,40){\scalebox{.25}{\usebox{\Rectref}}}
\put(0,60){\scalebox{.25}{\usebox{\Rectref}}}
\put(30,60){\scalebox{.25}{\usebox{\Rectref}}}
\put(60,60){\scalebox{.25}{\usebox{\Rectref}}}
\put(90,60){\scalebox{.25}{\usebox{\Rectref}}}
%%Rect5
\put(190,0){\usebox{\Rectfw}}
 \put(190,0){\scalebox{.2}{\usebox{\Rectref}}}
\put(214,0){\scalebox{.2}{\usebox{\Rectref}}}
\put(238,0){\scalebox{.2}{\usebox{\Rectref}}}
\put(262,0){\scalebox{.2}{\usebox{\Rectref}}}
 \put(286,0){\scalebox{.2}{\usebox{\Rectref}}}
\put(190,16){\scalebox{.2}{\usebox{\Rectref}}}
\put(214,16){\scalebox{.2}{\usebox{\Rectref}}}
\put(238,16){\scalebox{.2}{\usebox{\Rectref}}}
\put(262,16){\scalebox{.2}{\usebox{\Rectref}}}
 \put(286,16){\scalebox{.2}{\usebox{\Rectref}}}
\put(190,32){\scalebox{.2}{\usebox{\Rectref}}}
\put(214,32){\scalebox{.2}{\usebox{\Rectref}}}
\put(238,32){\scalebox{.2}{\usebox{\Rectref}}}
\put(262,32){\scalebox{.2}{\usebox{\Rectref}}}
 \put(286,32){\scalebox{.2}{\usebox{\Rectref}}}
\put(190,48){\scalebox{.2}{\usebox{\Rectref}}}
\put(214,48){\scalebox{.2}{\usebox{\Rectref}}}
\put(238,48){\scalebox{.2}{\usebox{\Rectref}}}
\put(262,48){\scalebox{.2}{\usebox{\Rectref}}}
 \put(286,48){\scalebox{.2}{\usebox{\Rectref}}}
\put(190,64){\scalebox{.2}{\usebox{\Rectref}}}
\put(214,64){\scalebox{.2}{\usebox{\Rectref}}}
\put(238,64){\scalebox{.2}{\usebox{\Rectref}}}
\put(262,64){\scalebox{.2}{\usebox{\Rectref}}}
 \put(286,64){\scalebox{.2}{\usebox{\Rectref}}}
%
%% Rect6
 \put(380,0){\usebox{\Rectfw}}
 \put(380,0){\scalebox{.1666}{\usebox{\Rectref}}}
\put(400,0){\scalebox{.1666}{\usebox{\Rectref}}}
\put(420,0){\scalebox{.1666}{\usebox{\Rectref}}}
\put(440,0){\scalebox{.1666}{\usebox{\Rectref}}}
 \put(460,0){\scalebox{.1666}{\usebox{\Rectref}}}
\put(480,0){\scalebox{.1666}{\usebox{\Rectref}}}
 \put(380,13.3333){\scalebox{.1666}{\usebox{\Rectref}}}
\put(400,13.3333){\scalebox{.1666}{\usebox{\Rectref}}}
\put(420,13.3333){\scalebox{.1666}{\usebox{\Rectref}}}
\put(440,13.3333){\scalebox{.1666}{\usebox{\Rectref}}}
 \put(460,13.3333){\scalebox{.1666}{\usebox{\Rectref}}}
\put(480,13.3333){\scalebox{.1666}{\usebox{\Rectref}}}
 \put(380,26.6666){\scalebox{.1666}{\usebox{\Rectref}}}
\put(400,26.6666){\scalebox{.1666}{\usebox{\Rectref}}}
\put(420,26.6666){\scalebox{.1666}{\usebox{\Rectref}}}
\put(440,26.6666){\scalebox{.1666}{\usebox{\Rectref}}}
 \put(460,26.6666){\scalebox{.1666}{\usebox{\Rectref}}}
\put(480,26.6666){\scalebox{.1666}{\usebox{\Rectref}}}
 \put(380,40){\scalebox{.1666}{\usebox{\Rectref}}}
\put(400,40){\scalebox{.1666}{\usebox{\Rectref}}}
\put(420,40){\scalebox{.1666}{\usebox{\Rectref}}}
\put(440,40){\scalebox{.1666}{\usebox{\Rectref}}}
 \put(460,40){\scalebox{.1666}{\usebox{\Rectref}}}
\put(480,40){\scalebox{.1666}{\usebox{\Rectref}}}
 \put(380,53.3333){\scalebox{.1666}{\usebox{\Rectref}}}
\put(400,53.3333){\scalebox{.1666}{\usebox{\Rectref}}}
\put(420,53.3333){\scalebox{.1666}{\usebox{\Rectref}}}
\put(440,53.3333){\scalebox{.1666}{\usebox{\Rectref}}}
 \put(460,53.3333){\scalebox{.1666}{\usebox{\Rectref}}}
\put(480,53.3333){\scalebox{.1666}{\usebox{\Rectref}}}
 \put(380,66.6666){\scalebox{.1666}{\usebox{\Rectref}}}
\put(400,66.6666){\scalebox{.1666}{\usebox{\Rectref}}}
\put(420,66.6666){\scalebox{.1666}{\usebox{\Rectref}}}
\put(440,66.6666){\scalebox{.1666}{\usebox{\Rectref}}}
 \put(460,66.6666){\scalebox{.1666}{\usebox{\Rectref}}}
\put(480,66.6666){\scalebox{.1666}{\usebox{\Rectref}}}
\put(40,85){$N=4$}%
\put(230,85){$N=5$}%
\put(420,85){$N=6$}%
%\put(420,85){\blue $S=6$}
%
\linethickness{2pt}%
%{\color{blue}%
%% \Dotline(380,0)(400,0){1}%(400,13.3333)(380,13.3333)(380,0){2}
%% \Dotline(400,0)(400,13.3333){1}%
%% \Dotline(400,13.3333)(380,13.3333){1}%
%% \Dotline(380,13.3333)(380,0){1}%
%\Dashline(380,0)(400,0){2}%(400,13.3333)(380,13.3333)(380,0){2}
%\Dashline(400,0)(400,13.3333){2}%
%\Dashline(400,13.3333)(380,13.3333){2}%
%\Dashline(380,13.3333)(380,0){2}%
%
%\Dashline(460,0)(500,0){2}%
%\Dashline(500,0)(500,26.6666){2}%
%\Dashline(500,26.6666)(460,26.6666){2}%
%\Dashline(460,26.6666)(460,0){2}%
%}
\end{picture}
\end{center}
\caption{The actuators supports for  $d=2$.} \label{fig.suppAc}
\end{figure}
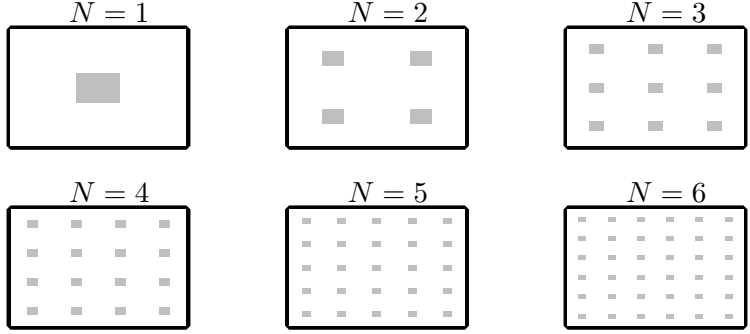
We have the  following characterization from  \cite[Lemma 3.8]{RodSturm20}   for the adjoint of the oblique projection.
\begin{lemma}
Suppose that $F$ and $G$ are closed subspaces of $H$,  for which  $H=F \oplus G$ holds.  Then for the adjoint operator of  $P^{G^{\perp}}_F \in \mathcal{L}(H)$,  we have  $(P^{G^{\perp}}_F)^* =P^{F^{\perp}}_G$.
\end{lemma}
In the next theorem,  we investigate the stabilizability of the following control system
\begin{equation}
\label{e9}
\begin{cases}
\partial_t y- \nabla \cdot ( \nu(\omega)\nabla y)+ a(t,x)y+  \nabla \cdot (b(t)y) = \sum^{N}_{i =1} u_i(t)\mathbf{1}_{O_i}   &   (t,x) \in  (t_0,\infty)\times D,\\
y =0   & (t,x)  \in  (t_0,\infty) \times \partial D,\\
y(t_0)=y_0(\omega) &  x \in D,
\end{cases}
\end{equation}
for almost surly $\omega \in \Omega$.
\begin{theorem}[Uniform stabilizability of \eqref{e9}]
\label{Theo1}

For each  $\mu>0$, there exists an integer $N^* \in \mathbb{N}_0$ such that for every  $N\geq N^*$ there exists  a feedback control vector $\bar{ \mathbf{u}}(y) = (\bar{u}_1, \dots, \bar{u}_N) \in L^2((t_0,\infty);U^{N}_{\mathbb{P}} ) \cong  L_{\mathbb{P}}^2(\Omega; \mathbb{R}) \otimes L^2((t_0,\infty); \mathbb{R}^N)$ for system \eqref{e9} whose associated state satisfies
\begin{equation}
\label{e13}
\| y(t)\|^{2}_{H_{\mathbb{P}}} \leq e^{-\mu (t-t_0)} \|y_0\|^2_{H_{\mathbb{P}}}  \quad  \text{ for all  }  t>0,
\end{equation}
for any given  $(t_0,y_0) \in \mathbb{R}_{\geq 0} \times H_{\mathbb{P}}$.
\end{theorem}
\begin{proof}
We set as the feedback control law
\begin{equation}
\label{e10}
    \sum_{i=1}^N \bar{u}_i(t,\omega) \mathbf{1}_{O_i}
    :=  - \lambda P_{\mathcal{O}_{N}}^{ \mathcal{E}_{N}^\perp}  \Delta P_{ \mathcal{E}_{N}}^ { \mathcal{O}^\perp_{N}}y(t,\omega) \quad \text{ for a.e. } t>0  \text{ a.s. } \omega \in \Omega,
\end{equation}
with
\begin{equation}
\label{e10a}
\mathbf{\bar{u}}(y)=(\bar{u}_1(y), \dots, \bar{u}_N(y))^t := -\lambda\mathcal{I} P_{\mathcal{O}_{N}}^{ \mathcal{E}_{N}^\perp}  \Delta P_{ \mathcal{E}_{N}}^ { \mathcal{O}^\perp_{N}}y(t,\omega),
\end{equation}
where  $\mathcal{I}: \mathcal{O}_N \to  \mathbb{R}^N$ stands for  the canonical isomorphism,  and  the numbers $\lambda>0$ and  $N \in \mathbb{N}_0$ are specified below.  Inserting \eqref{e10} in equation \eqref{e9},  multiplying with $y$,  and integrating over $D$,  we obtain for almost every $t>0$ and  almost surely  $\omega \in \Omega$ that
 \begin{equation}
 \label{e11}
\begin{split}
&\frac{d}{2dt} \|y(t, \omega)\|^2_H+   ( \nu(t, \omega) \nabla y(t,\omega),  \nabla y(t,\omega) )_{H}+   \langle a(t)y(t,\omega),y(t,\omega)\rangle_{V',V} \\& -  \langle  y(t,\omega), b(t)\cdot \nabla y(t,\omega) \rangle_{V',V} + \lambda  \langle P_{\mathcal{O}_{N}}^{ \mathcal{E}_{N}^\perp}  \Delta P_{ \mathcal{E}_{N}}^ { \mathcal{O}^\perp_{N}}y(t,\omega), y(t,\omega) \rangle_{V,V'}= 0
\end{split}
\end{equation}
From now on, we omit $\omega$ for simplicity,  i.e.  $y(t,\omega) = y(t)$.  From \eqref{e11} and using A1, it follows that
\begin{equation*}
\begin{split}
&\frac{d}{dt} \|y(t)\|^2_H  \leq  - 2\nu_{\min}(\omega) \| y(t) \|^2_V +2 | \langle a(t)y(t),y(t) \rangle_{V',V}|\\& + 2|(b(t)y(t), \nabla y(t))_H| + 2 \lambda  \langle P_{\mathcal{O}_{N}}^{ \mathcal{E}_{N}^\perp}  \Delta P_{ \mathcal{E}_{N}}^ { \mathcal{O}^\perp_{N}}y(t,\omega), y(t,\omega) \rangle_{V',V}.
\end{split}
\end{equation*}
We use the following decomposition
\begin{equation*}
y = \theta+ \varphi  \quad  \text{ with }  \theta := P_{ \mathcal{E}_{N}}^ { \mathcal{O}^\perp_{N}}y  \quad  \text{ and } \quad  \varphi := P^{ \mathcal{E}_{N}}_{ \mathcal{O}^\perp_{N}}y,
\end{equation*}
which is justified due to the definition of the oblique projection.  Further, since
$\theta  \in  \mathcal{E}_{N} \subset V$,  then $\Delta \theta \in V'$.  Thus,  the operator $P_{\mathcal{O}_{N}}^{ \mathcal{E}_{N}^\perp}$ can be considered as its unique linear extension to $V'$.  That is  $P_{\mathcal{O}_{N}}^{ \mathcal{E}_{N}^\perp}  \Delta \theta \in \mathcal{O}_N \subset H \subset V'$ and we have
\begin{equation*}
\langle P_{\mathcal{O}_{N}}^{ \mathcal{E}_{N}^\perp}  \Delta \theta, w \rangle_{V',V} =  \langle  \Delta \theta,    P_{\mathcal{E}_{N}}^{ \mathcal{O}_{N}^\perp}  w \rangle_{V',V}    \quad \text{ for all } w \in V,
\end{equation*}
which is well-defined due  the fact that  $P_{\mathcal{E}_{N}}^{ \mathcal{O}_{N}^\perp}  w \in \mathcal{E}_N \subset V$.   Thus,  we can write
\begin{equation}
\label{e42}
-\langle P_{\mathcal{O}_{N}}^{ \mathcal{E}_{N}^\perp}  \Delta \theta, y \rangle_{V',V} =  -\langle  \Delta \theta,   \theta \rangle_{V',V}  = \| \theta  \|^2_V.
\end{equation}
From \eqref{e11} and \eqref{e42},  its follows by repeated use  of  Young's inequality that
\begin{equation}
\label{e45}
\begin{split}
&\frac{d}{dt} \|y(t)\|^2_H \\ &  \leq  - 2\nu_{\min}(\omega) \| y(t) \|^2_V +2 | \langle a(t)y(t),y(t) \rangle_{V',V}| + 2|(b(t)y(t), \nabla y(t))_H| - 2\lambda \|\theta (t)\|^2_V\\
&\leq  - 2\nu_{\min}(\omega) \| y(t) \|^2_V+ 2c \mathcal{N}(a,b)\|y(t)\|_H\|y(t)\|_V  - 2\lambda \|\theta(t)\|^2_V\\
&\leq  -\nu_{\min}(\omega) \| \theta(t) + \varphi(t) \|^2_V+ \frac{c^2\mathcal{N}^2(a,b)}{\nu_{\min}(\omega)} \| \theta(t) + \varphi(t)\|^2_H-2\lambda \|\theta(t) \|^2_V,\\
&\leq  -\nu_{\min}(\omega) \left( \| \theta(t)\|^2_V + \|\varphi(t) \|^2_V\right)+\nu_{\min}(\omega) \left(\kappa_1 \| \theta(t)\|^2_V + \frac{1}{\kappa_1} \|\varphi(t) \|^2_V\right)\\
&+\frac{c^2\mathcal{N}^2(a,b)}{\nu_{\min}(\omega)} \left( \| \theta(t)\|^2_H + \|\varphi(t)\|^2_H\right)+ \frac{c^2\mathcal{N}^2(a,b)}{\nu_{\min}(\omega)} \left( \kappa_2 \| \theta(t)\|^2_H +\frac{1}{\kappa_2} \|\varphi(t)\|^2_H \right)-2\lambda \|\theta(t)\|^2_V,
\end{split}
\end{equation}
where $c$ is a generic constant that depends only on $D$,  and the numbers  $\kappa_1>0$ and $\kappa_2>0$ can be chosen arbitrary.    Setting $\kappa_1 =\kappa_2 = 2$ in the  above inequality,   we obtain
\begin{equation}
\label{e43}
\begin{split}
&\frac{d}{dt} \|y(t)\|^2_H \leq  - \left(2\lambda- \nu_{\min}(\omega) \right) \| \theta(t)\|^2_V+ \frac{3c^2\mathcal{N}^2(a,b)}{\nu_{\min}(\omega)} \| \theta(t)\|^2_H\\
&- \frac{\nu_{\min}(\omega)}{2}\| \varphi(t)\|^2_V+ \frac{3c^2\mathcal{N}^2(a,b)}{2\nu_{\min}(\omega)}\| \varphi(t)\|^2_H \leq -\Theta_{\theta}(N,\lambda)\|\theta(t)\|^2_H- \Theta_{\varphi}(N,\lambda)\|\varphi(t)\|^2_H,
\end{split}
\end{equation}
where the constants  $\Theta_{\theta}$ and $\Theta_{\varphi}$  are defined by
\begin{align}
\label{e18}
\Theta_{\theta}(\omega ,N,\lambda)&:=  \left(2\lambda- \nu_{\min}(\omega) \right)  \alpha_1- \frac{3c^2\mathcal{N}^2(a,b)}{\nu_{\min}(\omega)}, \\\label{e19}
 \Theta_{\varphi}(\omega,N,\lambda)&:= \frac{\nu_{\min}(\omega)}{2}\beta_N- \frac{3c^2\mathcal{N}^2(a,b)}{2\nu_{\min}(\omega)},
\end{align}
with  $\beta_N$ given in \eqref{e12}  and $\alpha_1$ as the smallest eigenvalue of the Laplacian with homogeneous Dirichlet  boundary conditions.

Choosing $N^*$ and $\lambda^*$  such that
\begin{equation}
\label{e15}
 \beta_{N^*} \geq \frac{2}{\underline{\nu}}\left(4\mu  + \frac{3c^2\mathcal{N}^2(a,b)}{2\underline{\nu}}  \right) \quad \text{ and } \quad
 \lambda^* \geq \frac{1}{2\alpha_1} \left( 4\mu+ \frac{3c^2\mathcal{N}^2(a,b)}{\underline{\nu}} \right) +\frac{\overline{\nu}}{2}.
\end{equation}
We can infer   for  every  $N\geq N^*$,   $\lambda \geq \lambda^* $  that $\Theta_{\theta}(N,\lambda)\geq 4\mu$  and  $\Theta_{\varphi}(N,\lambda) \geq 4\mu$.   Therefore,  together with \eqref{e43} we obtain
\begin{equation}
\label{e44}
\begin{split}
&\frac{d}{dt} \|y(t)\|^2_H \leq -\Theta_{\theta}(N,\lambda)\|\theta(t)\|^2_H- \Theta_{\varphi}(N,\lambda)\|\varphi(t)\|^2_H \leq  -4 \mu\|\theta(t)\|^2_H- 4\mu\|\varphi(t)\|^2_H \\ &\leq  -2\mu \left(  \|\theta(t)\|^2_H  +\|\varphi(t)\|^2_H + 2(\theta (t), \varphi (t))_H   \right)\leq -2\mu \| \theta(t)+\varphi(t) \|^2_H  \leq -2\mu \|y(t)\|^2_H,
\end{split}
\end{equation}
 for  a.e.   $t>t_0$  and  $\omega \in \Omega$ a.s..   Integrating \eqref{e44} over interval $(t_0, t)$ we obtain that
\begin{equation}
\label{e14}
\|  y(t,\omega) \|^2_H \leq  e^{-2\mu(t-t_0)} \|y(t_0,\omega)\|^2_H  =   e^{-2\mu(t-t_0)} \|y_0(\omega)\|^2_H     \quad \omega \in \Omega  \text{ for a.s. }.
\end{equation}
Finally,  integrating \eqref{e14}  over $\Omega$,  we obtain
\begin{equation*}
\|  y(t) \|^2_{H_{\mathbb{P}}} =  \mathbb{E} \left[ | y(t) \|^2_{H} \right]   \leq  e^{-2\mu(t-t_0)} \mathbb{E} \left[ \|y(t_0)\|^2_H \right]   =   e^{-2\mu(t-t_0)} \|y_0\|^2_{H_{\mathbb{P}}}.
\end{equation*}
This together with the fact that $ \bar{ \mathbf{u}} \in L^2((t_0,\infty);U^{N}_{\mathbb{P}})$ (see   \eqref{e10a}) completes the proof.
\end{proof}

\begin{remark}
\label{Rem1}
Assume that $b =0$ and $a \in L^{\infty}((0,\infty)\times D ; \mathbb{R})$.  Then,   by replacing   the term  $c \mathcal{N}(a,b)\|y(t)\|_H\|y(t)\|_V$ with $\| a\|_{L^{\infty}((0,\infty)\times D ; \mathbb{R})} \|y(t)\|^2_H$ in the third line of \eqref{e45} and following the same lines of computations as above,    \eqref{e18} and \eqref{e19} can be expressed as
\begin{align*}
\Theta_{\theta}(\omega ,N,\lambda)&:=  2\left(\lambda -\nu_{\min}(\omega) \right)  \alpha_1- 6  \|a\|_{ L^{\infty}((0,\infty)\times D ; \mathbb{R})}, \\
 \Theta_{\varphi}(\omega,N,\lambda)&:=\nu_{\min}(\omega) \beta_N- 3  \|a\|_{ L^{\infty}((0,\infty)\times D ; \mathbb{R})}.
\end{align*}
Hence,  for any given rate $\mu>0$,  the stabilizability result \eqref{e13} holds for $N^*$ and $\lambda^*$ satisfying
\begin{align}
\label{e39}
 \beta_{N^*} \geq \frac{1}{\underline{\nu}}\left(4\mu  + 3\|a\|_{ L^{\infty}((0,\infty)\times D ; \mathbb{R})}  \right)  \quad  \text{ and } \quad
 \lambda^* \geq \frac{1}{\alpha_1} \left( 2\mu+ 3\|a\|_{ L^{\infty}((0,\infty)\times D ; \mathbb{R})} \right)  +\overline{\nu}.
\end{align}
 \end{remark}
\subsection{Stability  of stochastic RHC}
In this section,  we investigate the stability of the receding horizon algorithm \ref{RRHA}.  Our theoretical results are expressed in terms of the finite- and infinite-horizon value functions and are based on the stability result given in the previous section.

We have the following stability result for the stochastic  RHC $ \mathbf{u}_{rh}$ obtained by Algorithm \ref{RRHA} with $U =U^{N}_{\mathbb{P}}$  and the choices of $\ell(t,y) = \| y\|^2_V$ and  $\ell(t,y)= \| y\|^2_H$ for almost every $t \in (0, \infty)$.
\begin{theorem}
\label{Theo3}
Suppose that $D \subset \mathbb{R}^d$ with $d \geq 1$ is a rectangle.  Then for every choice of  $\ell(t,\cdot) = \| \cdot \|^2_V$ or   $\ell(t,\cdot)= \| \cdot\|^2_H$,  there exits an $N^* = N^*(a,b,\nu) \in \mathbb{N}$ such the RHC computed by Algorithm \ref{RRHA} with $U:=U^{N}_{\mathbb{P}}$  is,  for every $N\geq N^*$ and set of actuators $\mathbf{1}_{O_i}$ with $i=1, \dots,N$ given in the previous section,  suboptimal and stabilizing.  That is  for every given $\delta>0$  there exist numbers $T^* > \delta$, and $\alpha \in (0,1)$,  such that for every fixed  prediction horizon $T \geq T^*$, and every $y_0 \in H_{\mathbb{P}}$ the control $\mathbf{u}_{rh} \in L^2((0,\infty);U^{N}_{\mathbb{P}} )$ provided Algorithm \ref{RRHA} by satisfies the suboptimality inequality
\begin{equation}
\label{ed27}
\alpha V_{\infty}(y_0) \leq \alpha J_{\infty}(\mathbf{u}_{rh};0,y_0)\leq V_T(0,y_0) \leq V_{\infty}(y_0).
\end{equation}
Furthermore,  we have
\begin{equation}
\label{e26}
\ \|y(t)\|^2_{H_{\mathbb{P}}}   \to 0     \quad  \text{ as }   t \to \infty,
\end{equation}
for the choice of $\ell(t,\cdot)= \| \cdot\|^2_H$,   and
\begin{equation}
\label{e26b}
\ \|y(t)\|^2_{H_{\mathbb{P}}} \leq e^{- \zeta t} c_e\|y_0\|^2_{H_{\mathbb{P}}}  \quad  \text{ for }   t \geq 0,
\end{equation}
for the choice of $\ell( t, \cdot)= \| \cdot\|^2_V$,   where $\zeta$ and $c_e$ are independent of $y_0$.
\end{theorem}
\begin{proof}
Algorithm \ref{RRHA} corresponds  to the receding horizon framework introduced in \cite{MR4022734} for time-varying linear evolution equations adapted to the spaces $V_{\mathbb{P}}\hookrightarrow H_{\mathbb{P}} \hookrightarrow V'_{\mathbb{P}} $.
The stability of this framework is based on the three key conditions which we will verify here.  The rest of the proof follows with the same arguments as given in  \cite[Theorem 2.6]{MR4022734}.

\textbf{P1:}   For every positive number $T$, $V_T$ is globally decrescent with respect to the $H_{\mathbb{P}}$-norm. That is, there exists a continuous,  non-decreasing,  and bounded function $\gamma_2: \mathbb{R}_{\geq 0} \to \mathbb{R}_{\geq 0}$ such that
\begin{equation}
\label{e71}
V_T( \bar{t}_0, \bar{y}_0)  \leq  \gamma_2(T)\| \bar{y}_0\|^2_{H_{\mathbb{P}}}  \quad \text{ for all }  (\bar{t}_0,\bar{y}_0)\in \mathbb{R}_{\geq 0}\times H_{\mathbb{P}}.
\end{equation}
It is sufficient to chose  $N^* \in \mathbb{N}$ as the smallest number for which 
\begin{equation*}
 \beta_{N^*} >  \frac{3c^2\mathcal{N}^2(a,b)}{\underline{\nu}^2}
\end{equation*}
 holds.  In this case,  for almost surly $\omega \in \Omega$,  we have
\begin{equation}
\label{e22}
\beta_{N^*}  > \frac{3c^2\mathcal{N}^2(a,b)}{\nu^2_{\min}(\omega)},
\end{equation}
and  we can use Theorem \ref{Theo1} to verify the stabilizability.   Indeed,  setting  $\bar{u} \in L^2((0,\infty); U^{N}_{\mathbb{P}})$ as in \eqref{e10} for any $N \geq N^*$  and $\lambda \geq \lambda^*$  with
\begin{equation}
\lambda^*  := \frac{1}{2\alpha_1} \left( 4\mu+ \frac{3c^2\mathcal{N}^2(a,b)}{\underline{\nu}} \right) +\frac{\overline{\nu}}{2},
\end{equation}
we obtain \eqref{e13} for the rate
\begin{equation*}
 \mu: = \frac{\underline{\nu}}{8} \left(  \beta_{N^*} - \frac{3c^2\mathcal{N}^2(a,b)}{\underline{\nu}^2}\right).
\end{equation*}
 Further,  for this control we can write
\begin{equation}
\label{e47}
\bar{\mathbf{u}}(\bar{y}(t,\omega))=(\bar{u}_1( \bar{y}(t,\omega)), \dots, \bar{u}_N(\bar{y}(t,\omega)))^T := -\lambda\mathcal{I}P_{\mathcal{O}_{N}}^{ \mathcal{E}_{N}^\perp}  \Delta P_{ \mathcal{E}_{N}}^ { \mathcal{O}^\perp_{N}} \bar{y}(t,\omega),
\end{equation}
where  $\mathcal{I}: \mathcal{O}_N \to  \mathbb{R}^N$ denotes the canonical isomorphism.    For the control  $\bar{\mathbf{u}}$ and its associated state $\bar{y} =y(\bar{\mathbf{u}})$ it holds that
\begin{equation*}
V_T(\bar{t}_0, \bar{y}_0) \leq  \frac{1}{2}\int_{\bar{t}_0}^{\bar{t}_0+T} \mathbb{E}\left[\ell(t, \bar{y}(t))\right]\,dt +\frac{\beta}{2}\int^{\bar{t}_0+T}_{\bar{t}_0} \mathbb{E}\left[ |  \bar{ \mathbf{u}}(t)|_{\ell_2}^2 \right] dt,
\end{equation*}
and,  depending on the choice of  $\ell$,  we have the following cases:

\textit{First case $\ell(t,\cdot)= \| \cdot\|^2_H$}:  Using the fact that
\[\| P_{\mathcal{O}_{N}}^{ \mathcal{E}_{N}^\perp}  \Delta P_{ \mathcal{E}_{N}}^ { \mathcal{O}^\perp_{N}} \|_{\mathcal{L}(H)} \leq c_P    \text{ and }   \|\mathcal{I} \|_{\mathcal{L}(\mathcal{O}_{N},\mathbb{R}^N)} \leq \hat{c}, \]
for positive constants $c_P$ and $\hat{c}$,  we obtain with  \eqref{e47} that
\begin{equation}
\label{e20}
\begin{split}
V_T&(\bar{t}_0, \bar{y}_0)  \leq \frac{1}{2}\int_{\bar{t}_0}^{\bar{t}_0+T} \left( \|\bar{y}(t)\|^2_{H_{\mathbb{P}}}\,dt + \lambda\beta \hat{c}^2 c^2_P \| \bar{y}(t)\|_{H_{\mathbb{P}}}^2 \right) \, dt \\ &  \leq  \frac{1+\lambda\beta \hat{c}^2 c^2_P}{2\mu}\left( 1-e^{-\mu T} \right)\|\bar{y}_0\|^2_{H_{\mathbb{P}}} =: \gamma_2(T) \| \bar{y}_0\|^2_{H_{\mathbb{P}}}.
\end{split}
\end{equation}
\textit{Second case $\ell(t,\cdot)= \| \cdot\|^2_V$}:  In this case,  with standard  energy estimates,  we have for almost every $t \in (\bar{t}_0, \bar{t}_0+T)$ and almost surly $\omega \in \Omega$ that
\begin{equation*}
\begin{split}
\frac{d}{2dt}&\| \bar{y}(t)\|^2_H + \underline{\nu} \| \bar{y}(t)\|^2_V  \leq   c \mathcal{N}(a,b)\|\bar{y}(t) \|_H\| \bar{y}(t)\|_V  +  \| \sum^N_{i =1} \bar{u}(t,\omega)\mathbf{1}_{O_i} \|_{H}\|\bar{y}(t)\|_{H}\\   &  \leq   \frac{1}{2} c_5 \|\bar{y}(t)\|^2_H + \frac{ \underline{\nu}}{2}\|\bar{y}(t)\|^2_V  + \frac{1}{2} | \bar{\mathbf{u}}(t)|_{\ell_2}^2,
\end{split}
\end{equation*}
where
%\begin{equation}
$c_5:=\left(\frac{c^2}{ \overline{\nu}}\mathcal{N}^2(a,b)+N\max_{ 1 \leq i  \leq N} \| \mathbf{1}_{O_i}\|^2_H \right).$
%\end{equation}
Integrating over $\Omega$ and  $(\bar{t}_0,\bar{t}_0+T)$,  together with  \eqref{e20},  we obtain
\begin{equation}
\label{e21}
\begin{split}
\int^{\bar{t}_0+T}_{\bar{t}_0}& \| \bar{y}(t)\|^2_{V_{\mathbb{P}}} dt \leq \frac{1}{\underline{\nu}}\left(\|\bar{y}(\bar{t}_0)\|^2_{H_{\mathbb{P}}}  +  c_5\int^{\bar{t}_0+T}_{\bar{t}_0} \|\bar{y}(t)\|^2_{H_{\mathbb{P}}} \,dt + \int^{\bar{t}_0+T}_{\bar{t}_0} \mathbb{E} \left[  | \bar{\mathbf{u}}(t)|^2_{\ell_2} \right]  \,dt \right)\\
&\leq \frac{1}{\underline{\nu}}\left(1 +\frac{c_5+\lambda \hat{c}^2 c^2_P}{\mu}(1-e^{-\mu T})  \right) \|\bar{y}_0\|^2_{H_{\mathbb{P}}}.
\end{split}
\end{equation}
Finally,  using \eqref{e21}, we have
\begin{equation*}
\begin{split}
V_T &(\bar{t}_0,\bar{y}_0)  \leq \frac{1}{2}\int_{\bar{t}_0}^{\bar{t}_0+T} \left( \|\bar{y}(t)\|^2_{V_{\mathbb{P}}}\,dt + \lambda\beta \hat{c}^2 c^2_P \| \bar{y}(t)\|_{H_{\mathbb{P}}}^2 \right) \, dt \\ &  \leq  \frac{1}{2\underline{\nu}}\left(1 +\frac{c_5+\lambda(1+\beta \underline{\nu}) \hat{c}^2 c^2_P}{\mu}(1-e^{-\mu T})  \right)\|\bar{y}_0\|^2_{H_{\mathbb{P}}} =: \gamma_2(T) \|\bar{y}_0\|^2_{H_{\mathbb{P}}}.
\end{split}
\end{equation*}

\textbf{P2:}  For every $(\bar{t}_0, \bar{y}_0) \in \mathbb{R}_0 \times H_{\mathbb{P}}$,  every finite horizon optimal control problem of the form  \ref{optT} admits a solution:

  The objective function  $J_{T}(\mathbf{u};\bar{t}_0,y_0)$ is strictly convex,  coercive,  and nonnegative.  Hence it is weakly lower semi-continuous and existence of a unique minimizer to  \ref{optT} follows from the direct method in the calculus of variations, see e.g., \cite[Theorem 1.43]{MR2516528}.

Since  P1 and P2 hold,  we are in the position that we can apply \cite[Theorem 6.2]{MR4022734} and thus \eqref{ed27} holds.  It remains now to show that \eqref{e26} and \eqref{e26b} are satisfied.

First,  we deal with \eqref{e26b}.   This follows using the same arguments given in the second part of  \cite[Theorem 6.2]{MR4022734} together with Property P3 stating:

\textbf{P3:}  For every $T>0$,  $V_T$ is uniformly positive with respect to the $H_{\mathbb{P}}$-norm.  In other words,  for every $T>0$ there exists a constant $\gamma_1(T)>0$ such that we have
\begin{equation}
\label{e27}
V_T(\bar{t}_0,\bar{y}_0)  \geq \gamma_1(T)\|\bar{y}_0\|^2_{H_{\mathbb{P}}}  \quad \text{ for every } (\bar{t}_0,\bar{y}_0)\in \mathbb{R}_{\geq 0} \times H_{\mathbb{P}}.
\end{equation}
We will next verify this property.   For any arbitrary given  $(\bar{t}_0,\bar{y}_0) \in \mathbb{R}_{\geq 0} \times H_{\mathbb{P}}$ and control $\mathbf{u} \in L^2((\bar{t}_0, \bar{t}_0+T); U_{\mathbb{P}}^N)$,  we have by \eqref{e6} that
 \begin{equation*}
\begin{split}
\|\bar{y}_0\|^2_{H_{\mathbb{P}}}  &\leq  c_{2}(1 + T^{-1} +  \mathcal{N}(a,b)) \int^{\bar{t}_0+T}_{\bar{t}_0}\|\bar{y}(t)\|^2_{V_{\mathbb{P}}}dt+ \int^{\bar{t}_0+T}_{\bar{t}_0}\|\sum^{N}_{i =1}\bar{u}_i \mathbf{1}_{O_i} \|^2_{V'_{\mathbb{P}}}dt.
\end{split}
\end{equation*}
Together with the estimate
\begin{equation*}
\begin{split}
\int^{\bar{t}_0+T}_{\bar{t}_0}\|\sum^{N}_{i =1} \bar{u}_i \mathbf{1}_{O_i} \|^2_{V'_{\mathbb{P}}}dt &\leq  i_{H_{\mathbb{P}},V'_{\mathbb{P}}} \int^{\bar{t}_0+T}_{\bar{t}_0}\|\sum^{N}_{i =1} \bar{u}_i \mathbf{1}_{O_i} \|^2_{H_{\mathbb{P}}}dt  \\ &\leq i_{H_{\mathbb{P}},V'_{\mathbb{P}}} N\max_{ 1 \leq i  \leq N} \| \mathbf{1}_{O_i}\|^2_H  \int^{\bar{t}_0+T}_{\bar{t}_0} \mathbb{E} \left[ | \bar{\mathbf{u}}(t)|^2_{\ell_2} \right] \,dt,
\end{split}
\end{equation*}
we obtain \eqref{e27} with $\gamma_1(T) := \left( \max\left\{ 2c_{2}(1+T^{-1}+ \mathcal{N}(a,b) ),  \frac{2}{\beta}( i_{H_{\mathbb{P}},V'_{\mathbb{P}}} N\max_{ 1 \leq i  \leq N} \| \mathbf{1}_{O_i}\|^2_H ) \right \} \right)^{-1} $,
where $i_{H_{\mathbb{P}},V'_{\mathbb{P}}}$ is  the embedding constant from $H_{\mathbb{P}}$ into $V'_{\mathbb{P}}$.  Therefore P3 holds and this completes the verification of  \eqref{e26b}.

Next we prove that  \eqref{e26} holds.  Using \eqref{ed27} and \eqref{e71},  we can write
 \begin{equation}
 \label{e28}
\int^{\infty}_0\|y_{rh}(t)\|^2_{H_{\mathbb{P}}}\,dt  \leq \frac{2\gamma_2(T)}{\alpha} \|y_0\|^2_{H_{\mathbb{P}}}   \quad  \text{ and }    \quad  \int^{\infty}_0 \mathbb{E}\left[ |\mathbf{u}_{rh}(t)|^2_{\ell_2} \right]  \,dt  \leq \frac{2\gamma_2(T)}{\alpha \beta} \|y_0\|^2_{H_{\mathbb{P}}}.
\end{equation}
Further,  with standard energy estimate we have for every $t\geq t_0$ that
\begin{equation*}
\begin{split}
\|y_{rh}(t) \|^2_{H_{\mathbb{P}}}  & + \underline{\nu} \int^{t}_{0} \|y_{rh}(t)\|^2_{V_{\mathbb{P}}}\,dt  \leq \|y_0\|^2_{H_{\mathbb{P}}}  +  \left(\frac{c^2}{\underline{\nu}}\mathcal{N}^2(a,b)+N\max_{ 1 \leq i  \leq N} \| \mathbf{1}_{O_i}\|^2_{H} \right)\int^{\infty}_{0} \|y_{rh}(t)\|^2_{H_{\mathbb{P}}}\,dt \\& + \int^{\infty}_{0} \mathbb{E}\left[|  \mathbf{u}_{rh}(t)|^2_{\ell_2} \right]\,dt \leq c_6\|y_0\|^2_{H_{\mathbb{P}}},
\end{split}
\end{equation*}
where  $c_6 :=\left(1 +\frac{ 2(1+\beta c_5)\gamma_2(T)}{\alpha\beta}\right)$ with
$c_5:= \left(\frac{c^2}{\underline{\nu}}\mathcal{N}^2(a,b)+N\max_{ 1 \leq i  \leq N} \| \mathbf{1}_{O_i}\|^2_{H} \right).$ Thus,  we can conclude that
\begin{equation}
\label{e78}
\|y_{rh}\|^2_{L^{\infty}((0,\infty);H_{\mathbb{P}})} \leq  c_6 \|y_0\|^2_{H_{\mathbb{P}}}    \quad  \text{ and }    \quad   \int^{\infty}_{0} \|y_{rh}(t)\|^2_{V_{\mathbb{P}}} \leq \frac{c_6}{\underline{\nu}} \|y_0\|^2_{H_{\mathbb{P}}}.
\end{equation}
Further, we have for every $t'' \geq t' \geq 0$ that
{\small
\begin{equation}
\begin{split}
\label{e77}
&\|y_{rh}(t'')\|^2_{H_{\mathbb{P}}}-\|y_{rh}(t')\|^2_{H_{\mathbb{P}}}=\int^{t''}_{t'}\frac{d}{dt}\|y_{rh}(t)\|^2_{H_{\mathbb{P}}}dt,\\
&=2\int^{t''}_{t'}  \mathbb{E}\left[ \langle  y_{rh}(t), \nu\Delta y_{rh}(t) - a(t)y_{rh}(t)-\nabla \cdot (b(t)y_{rh}(t))+ \sum^{N}_{i =1} (u_{rh})_i(t) \mathbf{1}_{O_i}\rangle_{V,V'}\right] dt,\\
&\leq-2\underline{\nu}\int^{t''}_{t'}\|y_{rh}(t)\|^2_{V_{\mathbb{P}}}dt\\&+2\int^{t''}_{t'}  \mathbb{E}\left[\langle - a(t)y_{rh}(t)-\nabla \cdot (b(t)y_{rh}(t))+ \sum^{N}_{i =1} (u_{rh})_i(t) \mathbf{1}_{O_i},y_{rh}(t)\rangle_{V,V'} \right] dt,\\
&\leq 2 \mathcal{N}(a,b)\int^{t''}_{t'}\|y_{rh}(t)\|_{V_{\mathbb{P}}}\|y_{rh}(t)\|_{H_{\mathbb{P}}}dt\\&
+2(N\max_{ 1 \leq i  \leq N} \| \mathbf{1}_{O_i}\|^2_{H})^{\frac{1}{2}}\int^{t''}_{t'} \| \mathbf{u}_{rh}(t)\|_{L_{\mathbb{P}}^2(\Omega; \mathbb{R}^N)}   \|y_{rh}(t)\|_{H_{\mathbb{P}}}dt\\
&\leq 2 \mathcal{N}(a,b)\big( \int^{t''}_{t'}\|y_{rh}(t)\|^2_{V_{\mathbb{P}}}dt\big)^{\frac{1}{2}}\big( \int^{t''}_{t'}\|y_{rh}(t)\|^2_{H_{\mathbb{P}}}dt\big)^{\frac{1}{2}}  \\&+2(N\max_{ 1 \leq i  \leq N} \| \mathbf{1}_{O_i}\|^2_{H})^{\frac{1}{2}}\big( \int^{t''}_{t'} \mathbb{E} \left[ |\mathbf{u}_{rh}(t)|^2_{\ell_2}\right]dt\big)^{\frac{1}{2}}\big( \int^{t''}_{t'}\|y_{rh}(t)\|^2_{H_{\mathbb{P}}}dt\big)^{\frac{1}{2}}\leq   c_7 \|y_0\|^2_{H_{\mathbb{P}}}(t''-t')^{\frac{1}{2}},
\end{split}
\end{equation}}
where  $c_7:= 2\left( \mathcal{N}(a,b){\underline{\nu}}^{\frac{-1}{2}}c_{6}  +  (N\max_{ 1 \leq i  \leq N} \| \mathbf{1}_{O_i}\|^2_{H})^{\frac{1}{2}}c^{\frac{1}{2}}_6 \left(\frac{2\gamma_2(T)}{\alpha \beta}\right)^{\frac{1}{2}} \right)$ and \eqref{e28}  and  \eqref{e78} were used in  the last inequality.

The rest of proof follows  the same lines as in  the proof of  \cite[Theorem 6.4]{MR4022734}  based on   \eqref{e77}  and \eqref{e28}.
\end{proof}

\subsection{Failure Probability}
\label{Sec5}
In this section,   we are concerned with the failure probability for the receding horizon framework.  For a given number of actuators $\bar{N}$,  we compute an upper bound for the probability of the case,  in which the stabilizability of  the stochastic RHC  computed by Algorithm \ref{RRHA} with($U:=U^N_{\mathbb{P}}$) and the deterministic variant of Algorithm \ref{RRHA} (\cite[Algorithm 1.1]{MR4022734}) with control  $U:=\mathbb{R}^N$  are not guaranteed.

Concretely,  let  $ \bar{N} \in \mathbb{N}_0$  be a given number of actuators.   Recalling the proof of Theorem \ref{Theo3},  it can be seen that the condition P1 and,  particularity,  inequality \eqref{e22} are essential.  Therefore,   P1 and the stabilizability of the controlled system may fail if 
\begin{equation}
\label{e46}
\nu^2_{\min}(\omega) \beta_{ \bar{N}}    \leq   3c^2\mathcal{N}^2(a,b).
\end{equation}
In this case,  for a given $y_0 \in H$,  the existence of a stabilizing deterministic  control  $ \mathbf{u}_{rh}(\omega) \in L^2((0,\infty);\mathbb{R}^N)$ with respect to $H$-norm  which is suboptimal  in the sense of \eqref{ed27} for 
\begin{equation}
\label{e50}
 \min_{ \mathbf{u} \in L^2((0,\infty);\mathbb{R}^{\bar{N}})} J^{\omega}_{\infty}(\mathbf{u};0,y_0) := \frac{1}{2}\int_{0}^{\infty}\ell(t, y(t))dt+\frac{\beta}{2}\int_{0}^{\infty} |\mathbf{u}(t)|_{\ell_2}^2dt,
\end{equation}
and also,  for any given $y_0 \in H_{\mathbb{P}}$  the existence a stabilizing stochastic control  $ \mathbf{u}_{rh} \in L^2((0,\infty); U_{\mathbb{P}}^{\bar{N}})$ with respect to $H_{\mathbb{P}}$-norm  which is suboptimal for
\begin{equation}
\label{e51}
 \min_{ \mathbf{u} \in L^2((0,\infty);U_{\mathbb{P}}^{\bar{N}})} J_{\infty}(\mathbf{u};0,y_0) = \frac{1}{2}\int_{0}^{\infty} \mathbb{E}\left[ \ell(t, y(t))\right]dt+\frac{\beta}{2}\int_{0}^{\infty}\mathbb{E}\left[ |\mathbf{u}(t)|_{\ell_2}^2 \right]dt,
\end{equation}
are not guaranteed for either of the choices of   $\ell(t,y) = \| y\|^2_V$ and  $\ell(t,y)= \| y\|^2_H$.    Therefore,   the failure probability,   for both of  the above problem formulations,  can be expressed by
\begin{equation}
\begin{split}
\label{e24}
\mathbb{P} \left[\nu^2_{\min}(\omega) \beta_{ \bar{N}}    \leq   3c^2\mathcal{N}^2(a,b)\right] = \mathbb{P}\left[\nu_{\min}(\omega) \leq  3^{\frac{1}{2}}c \mathcal{N}(a,b)\beta^{-\frac{1}{2}}_{ \bar{N}} \right].
\end{split}
\end{equation}
We consider now  both Examples \ref{Examp1} and \ref{Examp2}.

\textbf{Example \ref{Examp1}:} Setting   $\Gamma(\omega):=\sum^{M}_{j= 1}\abs{z_j (\omega)}\| \psi_j\|_{L^{\infty}(D)}$,   we have
\begin{equation}
\label{e25}
\begin{split}
 &\mathbb{P}\left[\nu_{\min}(\omega) \leq 3^{\frac{1}{2}}c \mathcal{N}(a,b)\beta^{-\frac{1}{2}}_{ \bar{N}} \right]  =  \mathbb{P}\left[\exp( -\Gamma(\omega)) \leq  3^{\frac{1}{2}}c \mathcal{N}(a,b)\beta^{-\frac{1}{2}}_{ \bar{N}}-\underline{\nu} \right].
\end{split}
\end{equation}
where  $ 3^{\frac{1}{2}}c \mathcal{N}(a,b)\beta^{-\frac{1}{2}}_{ \bar{N}}-\underline{\nu}\geq 0$ since otherwise \eqref{e46}  is not valid.
Using  \eqref{e12},  \eqref{e24}, \eqref{e25},  and Markov's inequality,  we get
\begin{equation*}
\begin{split}
 &\mathbb{P}\left[\nu_{\min}(\omega) \leq 3^{\frac{1}{2}}c \mathcal{N}(a,b)\beta^{-\frac{1}{2}}_{ \bar{N}} \right]  = \mathbb{P}\left[ - \Gamma(\omega) \leq \log \left(3^{\frac{1}{2}}c \mathcal{N}(a,b)\beta^{-\frac{1}{2}}_{ \bar{N}}-\underline{\nu} \right) \right] \\ & = \mathbb{P}\left[ \Gamma(\omega) \geq \log \left(\frac{1}{3^{\frac{1}{2}}c \mathcal{N}(a,b)\beta^{-\frac{1}{2}}_{ \bar{N}}-\underline{\nu}}\right)    \right]  \leq  \mathbb{E}\left[  e^{ \Gamma} \right] e^{ \log \left(3^{\frac{1}{2}}c \mathcal{N}(a,b)\beta^{-\frac{1}{2}}_{ \bar{N}}-\underline{\nu} \right)} \\ &=   \mathbb{E}\left[  e^{  \Gamma } \right] \left(3^{\frac{1}{2}}c \mathcal{N}(a,b)\beta^{-\frac{1}{2}}_{ \bar{N}}-\underline{\nu} \right) \leq   \mathbb{E}\left[  e^{  \Gamma } \right] \left( \left( \frac{3}{c_{\beta}}\right)^{\frac{1}{2}}c \mathcal{N}(a,b)\bar{N}^{-1}-\underline{\nu} \right).
\end{split}
\end{equation*}
where $\mathbb{E}\left[  e^{  \Gamma } \right]$ is bounded  and without  loss of generality,  we  assumed that $  \beta^{\frac{1}{2}}_{\bar{N}} (1+\underline{\nu}) \geq 3^{\frac{1}{2}}c \mathcal{N}(a,b)$. Here we recall that $ \beta_N \to \infty $ as $N\to \infty$.

\textbf{Example \ref{Examp2}:}  Setting $\Gamma(\omega):=\sum^{\infty}_{j= 1}\abs{z_j (\omega)}\| \psi_j\|_{L^{\infty}(D)}$, we obtain
 \begin{equation}
 \label{e23}
\begin{split}
 &\mathbb{P}\left[\nu_{\min}(\omega) \leq 3^{\frac{1}{2}}c \mathcal{N}(a,b)\beta^{-\frac{1}{2}}_{ \bar{N}} \right] = \mathbb{P}\left[\Gamma(\omega) \geq \nu^* -3^{\frac{1}{2}}c \mathcal{N}(a,b)\beta^{-\frac{1}{2}}_{ \bar{N}}\right],
\end{split}
\end{equation}
 where, recalling that $ \beta_N \to \infty $ as $N\to \infty$,  we assume that  $\nu^* -3^{\frac{1}{2}}c \mathcal{N}(a,b)\beta^{-\frac{1}{2}}_{ \bar{N}}>0$.
%since otherwise we have for almost surly $\omega \in \Omega $
%\begin{equation}
% {\nu}^2_{\min}(\omega) \leq (\nu^*)^2 \leq 3c^2 \mathcal{N}^2(a,b)\beta^{-1}_{ \bar{N}}.
%\end{equation}
Using \eqref{e12},  \eqref{e24}, \eqref{e23},  and Markov's inequality,  we obtain
\begin{equation*}
\begin{split}
&\mathbb{P}\left[\nu_{\min}(\omega) \leq 3^{\frac{1}{2}}c \mathcal{N}(a,b)\beta^{-\frac{1}{2}}_{ \bar{N}} \right]  = \mathbb{P}\left[\Gamma(\omega)  \geq \nu^* -3^{\frac{1}{2}}c \mathcal{N}(a,b)\beta^{-\frac{1}{2}}_{ \bar{N}}\right] \\ & \leq  \frac{\mathbb{E}\left[  \Gamma  \right]}{\nu^* -3^{\frac{1}{2}}c \mathcal{N}(a,b)\beta^{-\frac{1}{2}}_{ \bar{N}}} \leq  \frac{ \sum_{j \geq 1}\mathbb{E}\left[ \abs{z_j} \right] \| \psi_j  \|_{L^{\infty}(D)}}{\nu^* - ( \frac{3}{c_{\beta}})^{\frac{1}{2}}c \mathcal{N}(a,b)\bar{N}^{-1}}.
\end{split}
\end{equation*}

\section{Parabolic PDEs with  log-normal diffusions}
In this section,  we study the case of log-normal diffusions defined by
\begin{equation}
\label{log-N2}
 \nu(\omega,x) = \exp(g(\omega,x))
\end{equation}
with $g$  a Gaussian random field with zero mean.  This class of diffusions is used in many applications, including those related to subsurface flow modeling and hydrology.   More precisely, for each $x \in D$,   $g(x, \cdot)$ is a Gaussian random variable,  and thus $0 < \nu(\omega,x)< \infty$
for each $\omega \in \Omega$. However, for any $\epsilon >0$ we have $\mathbb{P}\lbrack \nu(\cdot,x) > \epsilon^{-1} \rbrack>0$, and thus its corresponding elliptic operator is not uniformly bounded from above over all possible realizations of $\omega$.   We also have $\mathbb{P}\lbrack \nu(\cdot,x) < \epsilon \rbrack>0 $, so the corresponding elliptic operator  is not uniformly elliptic either.
%In this section,  we investigate the case of log-normal diffusions defined by
%\begin{equation}
%\label{log-N2}
% \nu(x,\omega) = \exp(g(x,\omega))
%\end{equation}
%where $g$ is a zero-mean Gaussian random field.  This class of diffusions is commonly used in many applications,  for example in  the context of subsurface flow modelling and hydrology.   To be more precise,  for each  $x \in D$,  $g(x, \cdot)$ is a Gaussian random variable,  and thus  $0 < \nu(\omega,x)< \infty$
%for any  $\omega \in \Omega$. However,  for any $\epsilon >0$,  we have  $\mathbb{P}\lbrack \nu(x, \cdot) > \epsilon^{-1} \rbrack>0 $,  and thus,  \eqref{log-N2} is not uniformly bounded from above over all possible realizations of $\omega$.   We also have $\mathbb{P}\lbrack \nu(x, \cdot) < \epsilon \rbrack>0 $ so that the operator  \eqref{log-N2} is not uniformly elliptic either.
\begin{assumption}
\label{assump2}
Throughout this section,   we impose the following conditions:
\begin{itemize}
\item[A1:]
There are random variables $\nu_{\min}$,  and  $\nu_{\max}$, such that
\begin{equation}
 0<\nu_{\min}(\omega)   \leq \nu(\omega,x) \leq  \nu_{\max}(\omega)  < \infty  \quad  \text{ for a.e.  }  x \in D  \text{ and a.s.  }  \omega \in \Omega,
\end{equation}
where $ \nu_{\max}(\omega) ,  \frac{1}{\nu_{\min}(\omega)}  \in  L^p_{\mathbb{P}}(\Omega; \mathbb{R})$ for $p \in [1, \infty)$.
\item[A2:] We also assume for $a$ and $b$ that
 \begin{equation}
 a \in L^{\infty}((0,\infty)\times D ; \mathbb{R})   \text{ and }  b =0.
 \end{equation}
 \end{itemize}
 \end{assumption}
As an example for the diffusion constant with the log-normal distribution satisfying Assumption \ref{assump2},  we refer to \cite{MR2649152,MR4246090,MR2805155}.
\begin{example}
\label{Examp3}
We set $g(\omega,x):=\sum^{\infty}_{j= 1}z_j(\omega) \psi_j(x)$ in \eqref{log-N2},
%Here  $\nu$ can be formally expressed as
%\begin{equation}
%\label{log-N-2}
% \nu(x,\omega) =\exp(\sum^{\infty}_{j= 1}z_j(\omega) \psi_j(x))
%\end{equation}
where the functions  $\psi_j \in L^{\infty}(D;\mathbb{R})$ for $j = 1,2,\dots$ are such that $\sum^{\infty}_{j = 1} \| \psi_j(x)\|_{L^{\infty}(D;\mathbb{R})}$ is finite and the random variables $z_j$ are i.i.d,  standard normal random variables, that is,  $z_j \sim \mathcal{N}(0,1)$ in $\mathbb{R}$.  For describing the resulting random field we set $\mathcal{F}:=\bigotimes^{\infty}_{j =1} \mathcal{B}(\mathbb{R})$,  where $\mathcal{B}(\mathbb{R})$ is the Borel $\sigma$-algebra in $\mathbb{R}$.  In this case,  the probability measure can be expressed  as the Gaussian product probability,  that is,  $\mathbb{P} := \bigotimes^{\infty}_{j =1} \mathcal{N}(0,1)$.    For the well-posedness of $\nu$ in \eqref{log-N2},  we restrict $z$  to be in the set $\Omega:=\{ z \in \mathbb{R^N}:  \sum^{\infty}_{j =1} |z_j|\|\psi_j \|_{L^{\infty}(D)}<\infty \}$.  In this case,   $\Omega$ is  $\mathcal{F}$-measurable,   and  $\mathbb{P}(\Omega) = 1$ holds.   Further,  for every $\omega \in \Omega$,  the following quantities are well-defined
\begin{equation}
\nu_{\max}(\omega) = \exp(\sum^{\infty}_{j= 1}\abs{z_j(\omega)}\| \psi_j\|_{L^{\infty}(D)}),  \qquad
\nu_{\min}(\omega)  = \exp( - \sum^{\infty}_{j= 1}\abs{z_j (\omega)}\| \psi_j\|_{L^{\infty}(D)}),
\end{equation}
and satisfy A1 in Assumptions \ref{assump2}.  See \cite{zbMATH06030212,MR2805155} for more details.
\end{example}
\subsection{Well-posedness of state equation}
For the log-normally distributed diffusion,  the existence is more delicate.  In fact,  due to  lack of integrability,   we can not  use  the weak formulation \eqref{e7} directly.  In this case,   first,  we show that the solution $y(\omega) \in  W(0,T)$  exists for  $\omega \in \Omega$ a.s.,   then we justify the measurability  and integrability of the mapping   $y: \Omega \to W(0,T)$ with $\omega \mapsto y(\omega)$.  The latter relies on controlling the integrability of all constants in the estimates.
\begin{theorem}
\label{Theo5}
Suppose that  Assumption \ref{assump2} holds.  Then for every given $(t_0,T,y_0, f) \in  \mathbb{R}^2_{>0} \times H \times L^2((t_0,t_0+T);H)$ equation \eqref{e17} admits a unique solution $y \in W_{\mathbb{P}}(t_0,t_0+T)$ satisfying the following estimates
\begin{equation}
\label{e31}
\|y\|^2_{C([t_0,t_0+T];H_{\mathbb{P}})} + \| y \|^2_{W_{\mathbb{P}}(t_0,t_0+T)} \leq c_3\left( \|y_0\|^2_{H} + \| f\|^2_{L^2((t_0,t_0+T);H)} \right), %\\ \label{e13}
\end{equation}
with $c_3$ depending on  $(T,a,b,D,\Omega)$.  Moreover, we have the following observability inequality
\begin{equation}
\label{e32}
\|y_0\|^2_{H} \leq c_{4}\|y\|^2_{L^2((t_0,t_0+T);V_{\mathbb{P}})}+ \|f\|^2_{L^2((t_0,t_0+T);H)},
\end{equation}
with  ${c}_4$ depending only on $(T,a,b,D,\Omega)$.
\end{theorem}
\begin{proof}
The proof is mainly inspired by the arguments  given  in \cite{MR2649152} for the well-posedness of the elliptic PDEs with log-normal diffusion.  It follows the following steps: First,  for each $\omega \in \Omega$ a.s.,  we consider  the unique solution $y(\omega) \in W(0,T)$ and derive the estimates for the resulting family of solutions.  Next,  we show that  $y(\omega)$ is measurable with respect to $\omega$.  Finally,  we show the integrability of this solution.  Throughout,  $c$ is a generic constant that is independent of $\omega$.

Using  standard arguments for deterministic parabolic PDEs,  it can be shown that for  $\omega \in \Omega$  a.s.,   \eqref{CS} admits a unique weak solution $y(\omega)$ with
\begin{equation}
\label{e35}
\begin{split}
\frac{d}{2dt} \| y(\omega,t)\|^2_{H}+\nu_{\min}(\omega) \| y(\omega,t)\|^2_V  &\leq \|a\|_{L^{\infty}((0,\infty)\times D ; \mathbb{R})}\|y(\omega,t)\|^2_H+ (f(t),y(\omega,t))_H\\
& \leq \frac{ \left(1+2\|a\|_{L^{\infty}((0,\infty)\times D ; \mathbb{R})} \right)}{2} \|y(\omega,t)\|^2_H+ \frac{1}{2} \|f(t)\|^2_H.
 \end{split}
\end{equation}
Using Gronwall's inequality we obtain for every $t  \in [t_0,t_0+T]$ that
\begin{equation*}
\begin{split}
\|y(\omega,t)\|^2_H&+2\nu_{\min}(\omega)\| y(\omega)\|^2_{L^2((t_0, t); V)} \\& \leq c \exp(T(1+2\|a\|_{L^{\infty}((0,\infty)\times D ; \mathbb{R})}))\left(\|y_0\|^2_H+\|f\|^2_{L^2((t_0,t_0+T);H)}   \right).
\end{split}
\end{equation*}
Therefore,  we can infer that
\begin{equation}
\label{e30}
\| y(\omega)\|^2_{L^2((t_0,t_0+T);V)} \leq \frac{c}{\nu_{\min}(\omega)}\left(\|y_0\|^2_H+\|f\|^2_{L^2((t_0,t_0+T);H)}   \right).
\end{equation}
Similarly,  we can write
\begin{equation*}
\begin{split}
\|\partial_ty(\omega)\|^2_{L^2((t_0,t_0+T);V')} &\leq c\left( \nu_{\max}^2(\omega)+ \|a\|^2_{L^{\infty}((0,\infty)\times D ; \mathbb{R})} \right)\| y(\omega)\|^2_{L^2((t_0,t_0+T);V)} + c\|f\|^2_{L^2((t_0,t_0+T);H)}\\
&\leq \frac{c(1+\nu_{\max}^2(\omega))}{\nu_{\min}(\omega)}\left(\|y_0\|^2_H+\|f\|^2_{L^2((t_0,t_0+T);H)}   \right)
\end{split}
\end{equation*}
and thus,  together with  \eqref{e30} we have
\begin{equation}
\label{e33}
\|y(\omega)\|^2_{W(0,T)} \leq c_c(\omega) \left(\|y_0\|^2_H+\|f\|^2_{L^2(t_0,t_0+T;H)}   \right),
\end{equation}
where $c_c(\omega):=\frac{c(1+\nu_{\max}^2(\omega))}{\nu_{\min}(\omega)}$.

Next,  we show the measurability of  the solution operator.  To do this,  we show that $y$ is a.s. the limit of a sequence of measurable functions.   One can adapt the proof given
in  \cite[Theorem 3.4]{MR2649152}.  Hence,  we just sketch its main idea here.  We first define,  for every $n  \in \mathbb{N}$,  the sequence  $\Omega_n \subset \Omega$ for which the diffusion coefficient is uniformly bounded
\begin{equation}
\Omega_n :=\{ \omega \in \Omega :   \nu_{\max}(\omega) < n  \text{ and } \nu_{\min}(\omega)>\frac{1}{n} \} \subset \Omega.
\end{equation}
Then  $\{ \Omega_n \}_n$ is increasing with $\Omega=\cup_{n\in\mathbb{N}} \Omega_n$.  Further,  invoking Theorem \ref{Theo2},  it follows that \eqref{CS} has the  unique solution   $y_n \in L^2_{\mathbb{P} |_{\Omega_n} }(\Omega; W(t_0,t_0+T))$ where  $\mathbb{P} |_{\Omega_n}$ stands for the restriction of  $\mathbb{P}$ to $\Omega_n$.
 This solution can be extended by zero to the solution defined on the whole of  $\Omega$,
  which  we denote again by $y_n$.   By definition,   $y_n$ is  measurable.  Further,  due to the fact that  $\Omega=\cup_{n\in\mathbb{N}} \Omega_n$, that  $y_n$  solves \eqref{CS} for $\omega \in \Omega_n$ a.s.,   and by uniqueness of the solution,    the function $y$ is the a.s.  limit of measurable functions $\{ y_n\}_n$,  and thus it is also measurable.
%
%the proof of the measurability can be adopted from \cite[Theorem 3.4]{MR2649152} . Here we just sketch its main idea.  We start with defining a subspace  $\Omega_n \subset \Omega$,  for every $n  \in \mathbb{N}$, where the diffusion coefficient is uniformly bounded
%\begin{equation}
%\Omega_n :=\{ \omega \in \Omega :   \nu_{\max}(\omega) < n  \text{ and } \nu_{\min}(\omega)>\frac{1}{n} \} \subset \Omega.
%\end{equation}
%Note that $\{ \Omega_n \}_n$ is increasing with $\Omega=\cap_{n\in\mathbb{N}} \Omega_n$.  In the uniform case,   from the Theorem \ref{Theo2},  it follows that there exists a unique solution $y_n \in L^2_{\mathbb{P}}(\Omega; w(t_0,t_0+T))$   . In particular,  $y_n$ is a measurable function on   $\Omega_n$ . The last step
%is proving that $y$ is a.s. limit of $y_n$, thus it is measurable. This follows because $y_n$ also solves the  equation \eqref{CS}  for a.s.  $\omega \in \Omega_n$
%

Now, we are in a position where we can show the integrability of the solution.  Using  A1  in Assumption \ref{assump2},  we  integrate over   $\Omega$ both of  sides of \eqref{e33},  and conclude that
\begin{equation*}
\|y\|^2_{L^p(\Omega;W(0,T))} \leq  \tilde{c}(p)\left(\|y_0\|^2_H+\|f\|^2_{L^2((t_0,t_0+T);H)}   \right),
\end{equation*}
 for every $p \in [1, \infty)$, and a constant  $\tilde{c}(p)>0$ depending on $p$,   using that  $c_c(\omega)$ is $p$-integrable.  In particular,  due to \eqref{e48},   inequality \eqref{e31} holds for $p =2$ with  a constant $c_3:=\tilde{c}(2)>0$.

 We next verify the observability estimate \eqref{e32}.    Multiplying \eqref{e17} by $\frac{T+t_0-t}{T}y(t)$,   integrating over $(t_0,t_0+T)$,  and using Young's inequality,  we obtain for   $\omega \in \Omega$ a.s.
\begin{equation}
\label{e29}
\begin{split}
&\|y_0 \|^2_H =\frac{1}{T} \int^{t_0+T}_{t_0}\| y(\omega, t) \|^2_H\,dt\\&+ 2\int^{t_0+T}_{t_0}\frac{t_0+T-t}{T}\Bigl( \nu_{\max}(\omega) \|y(\omega, t)\|^2_{V} + \langle a(t)y(\omega,t),y(\omega,t) \rangle_{V',V}-\langle f(t),y(\omega,t)\rangle_{V',V}  \Bigr)dt
\\& \leq c\left( T^{-1} +\nu_{\max}(\omega)+1+ \|a\|_{L^{\infty}((0,\infty)\times D ; \mathbb{R})} \right)\int^{t_0+T}_{t_0}\| y(\omega,t)\|^2_V\,dt +\int^{t_0+T}_{t_0}\| f(t)\|^2_{H}\,dt,
\end{split}
\end{equation}
where $c>0$ depends only on $D$.  Setting  $c_o(\omega):=c\left( T^{-1} +\nu_{\max}(\omega)+1+ \|a\|_{L^{\infty}((0,\infty)\times D ; \mathbb{R})} \right)$,  dividing \eqref{e29} by $c_o(\omega)$,  and integrating over $ \Omega$,  we obtain \eqref{e32} and the proof is complete.
\end{proof}

\subsection{ Stability of deterministic RHC}
According to Theorem \ref{Theo5},   the well-posedness
of the state in  $W_{\mathbb{P}}(0,T)$  for the log-normal diffusion,  is justified only for deterministic initial  and forcing functions. Due to the lack of integrability (see \eqref{e33}),   it is not clear how it can be extended for the random fields as initial and forcing functions.
 Therefore,   first,  at any time instances $t_i$  of the receding horizon framework,  we turn the random fields $y_{rh}( t_j, \cdot, \cdot)$  to a deterministic initial function $\bar{y}_0(\cdot) = \left( \mathbb{E} \left[y_{rh}(t_j, \cdot)^2\right] \right)^{\frac{1}{2}}$  by computing average of the squared function with respect to $\omega$.  Then we plug this deterministic initial function into the online open-loop problem.  In this regard,  we have modified and changed Algorithm \ref{RRHA}  to develop Algorithm \ref{RRHA_l}.   Further,  we need to restrict ourselves here to a deterministic control and, thus,  for every $ \bar{y}_0 \in H $ we consider the following performance index
\begin{equation}
\label{e36}
J_{T}(\mathbf{u};\bar{t}_0, \bar{y}_0):= \frac{1}{2}\int_{\bar{t}_0}^{\bar{t}_0+T} \mathbb{E}\left[ \| y(t)\|^2_V\right]\,dt +\frac{\beta}{2}\int^{\bar{t}_0+T}_{\bar{t}_0} |\mathbf{u}(t)|_{\ell_2}^2  dt.
\end{equation}
In the next theorem,  we investigate the stability of the control obtained by Algorithm \ref{RRHA_l} for $U:=\mathbb{R}^N$ and $\ell(t,y) := \|y\|^2_V$ .
\begin{algorithm}[htbp]
%\floatname{algorithm}{RHA}
\caption{Robust RHC($\delta,T$) for the log-normal diffusion }\label{RRHA_l}
%\begin{spacing}{1.1}
\begin{algorithmic}[5]
\REQUIRE{The sampling time $\delta$,  the prediction horizon $T\geq \delta$,   and the initial state $y_0$}
\ENSURE{The stability of the RHC~$\mathbf{u}_{rh}$.}\\
We proceed through the steps of Algorithm \ref{RRHA} except that Steps 1, 4,  and 5 are replaced by:
\STATE 1.  Compute~$\mathbb{E} \left[y^2_{0}(x)\right]$ for $x \in D$  and  set~$(\bar{t}_0,\bar{y}_0):=\left(0, \left( \mathbb{E} \left[ y^2_{0}\right] \right)^{\frac{1}{2}}\right)$  and  $y_{rh}(0) =y_0 $;
\STATE  4.   Compute   $ \mathbb{E} \left[y_{rh}(\bar{t}_0 +\delta,x)^2\right]$ of the state  for any  $x \in D$  at time $\bar{t}_0+ \delta$;
\STATE 5.  Update: $(\bar{t}_0,\bar{y}_0)  \leftarrow \left(\bar{t}_0 +\delta, \left( \mathbb{E} \left[y_{rh}(\bar{t}_0+\delta)^2 \right] \right)^{\frac{1}{2}}\right)$;
%  $y_{0}=y^*_T(t_0;\mathcal{I}_{0})$, $c_{0}=c^*_T(t_0;\mathcal{I}_{0})$,
%  and~$c_{0}^1=\dot c^*_T(t_0;\mathcal{I}_{0})$;
%\ENDWHILE
\end{algorithmic}
%\end{spacing}
\end{algorithm}

\begin{theorem}
\label{Theo4}
Suppose that Assumption  \ref{assump2} holds and
\begin{equation}
\label{e34}
\nu_{\min}(\omega)+ \essinf\{ a(t,x):   (t,x) \in (0,\infty)\times D  \}>0, \text{  for }      \omega \in \Omega \text{ a.s.}.
\end{equation}
Then,  for  $N\geq 1$  with the set of actuators $\{  \mathbf{1}_{O_i} \}^{N}_{i = 1}$  given in Section \ref{Sec3} ,  Algorithm \ref{RRHA_l} for  $\ell(t,\cdot): = \| \cdot \|^2_V$  is  suboptimal and stabilizing.  That is,   for every given $\delta>0$,  there exist numbers $T^* > \delta$, and $\alpha \in (0,1)$,  such that for every fixed  prediction horizon $T \geq T^*$, and every $y_0 \in H_{\mathbb{P}}$,  the RHC $\mathbf{u}_{rh} \in L^2((0,\infty);\mathbb{R}^N) $ satisfies the suboptimality inequality  \eqref{ed27}  and exponential stability result \eqref{e26b}.
\end{theorem}
\begin{proof}
The proof is similar to the one of Theorem \ref{Theo2} and is based on the arguments given in  \cite[Theorem 2.6]{MR4022734}.  To be more precise,  we  need again to verify the properties P1-P3 given in the proof of Theorem \ref{Theo2} with respect to the $H$-norm in place of  the $H_{\mathbb{P}}$-norm.   After verifying  P1-P3,  according to the construction (see Steps 1, 4, and 5) of Algorithm \ref{RRHA_l},  it can be easily shown that \eqref{e71} and  \eqref{e27} hold at every time instance $t_i = i\delta$ for $y_{rh}(t_i)$ with respect to the $H_{\mathbb{P}}$-norm,  that is,   $V_T(t_i,y_{rh}(t_i))  \leq  \gamma_2(T)\|y_{rh}(t_i)\|^2_{H_{\mathbb{P}}}$  and $V_T(t_i,y_{rh}(t_i))  \geq  \gamma_1(T)\|y_{rh}(t_i)\|^2_{H_{\mathbb{P}}}$  hold for every $i = 0,1,2,\dots$.  The rest of the proof can be completed along the routine of the proof of \cite[Theorem 2.6]{MR4022734}.  Therefore,  we will confine ourselves here to the justification of properties P1-P3.

 To verify  P1,   we set $\bar{\mathbf{u}}:= 0$ in \eqref{CS}  and define $\hat{\nu}(\omega) :=\nu_{\min}(\omega)+ \essinf\{ a(t,x):   (t,x) \in (0,\infty)\times D  \}$.   Then using the standard energy estimates,  we obtain for any $(\bar{t}_0, T) \in \mathbb{R}^2_{\geq 0}$    and   $\bar{y}_0 \in H$ that
\begin{equation*}
\frac{d}{2dt}\| \bar{y}(\omega,t)\|^2_H + \hat{\nu}(\omega)\|\bar{y}(\omega, t)\|^2_H  \leq 0  \quad    \text{ for almost every } t\geq \bar{t}_0,
\end{equation*}
Thus, we have for  $\omega \in \Omega$  a.s.  that
\begin{equation}
\label{e38}
\|\bar{y}(\omega,t)\|^2_H  \leq  e^{-2 \hat{\nu}(\omega)(t-\bar{t}_0)}  \|\bar{y}_0\|^2_H.
\end{equation}
Integrating over $(\bar{t}_0,  \bar{t}_0+T)$, we obtain
\begin{equation*}
\|\bar{y}(\omega)\|^2_{ L^2((\bar{t}_0,  \bar{t}_0+T);H)}  \leq \left(\frac{1-e^{-2\hat{\nu}(\omega)T}}{2 \hat{\nu}(\omega)}     \right)\|\bar{y}_0\|^2_H \leq  \frac{1}{2\hat{\nu}(\omega)}\|\bar{y}_0\|^2_H.
\end{equation*}
Thus,  similarly to \eqref{e35},  we can write
\begin{equation*}
\begin{split}
 \|\bar{y}(\omega)\|^2_{L^2((\bar{t}_0,  \bar{t}_0+T);V)}& \leq  \frac{1}{2\nu_{\min}(\omega)}  \|\bar{y}_0\|^2_{H}+ \frac{\|a\|_{L^{\infty}((0,\infty)\times D ; \mathbb{R})}}{\nu_{\min}(\omega)} \|\bar{y}(\omega)\|^2_{L^2((\bar{t}_0,  \bar{t}_0+T);H)}\\& \leq \left(\frac{1}{2\nu_{\min}(\omega)} + \frac{\|a\|_{L^{\infty}((0,\infty)\times D ; \mathbb{R})}}{2\hat{\nu}(\omega)\nu_{\min}(\omega)}\right) \|\bar{y}_0\|^2_{H}.
 \end{split}
\end{equation*}
Integrating over $\omega \in \Omega $,  we obtain
\begin{equation}
\label{e37}
\|\bar{y}\|^2_{L^2((\bar{t}_0,  \bar{t}_0+T);V_{\mathbb{P}}))} =\|\bar{y}\|^2_{L^2_{\mathbb{P}}(\Omega ;L^2((\bar{t}_0,  \bar{t}_0+T);V))} \leq c_s \|\bar{y}_0\|^2_H,
\end{equation}
 where the integrability of $\frac{1}{\nu_{\min}(\omega)}$ and $\frac{1}{\hat{\nu}(\omega)\nu_{\min}(\omega)}$  is justified due to \eqref{e34} and A1 in Assumption \ref{assump2}.  Setting $\bar{\mathbf{u}}:= 0$ in \eqref{e36}  and using  \eqref{e37},   we arrive at
\begin{equation*}
V_T(\bar{t}_0, \bar{y}_0) \leq  \frac{1}{2}\int_{ \bar{t}_0}^{ \bar{t}_0+T} \mathbb{E}\left[ \|\bar{y}(t)\|^2_V\right] \,dt +\frac{\beta}{2}\int^{ \bar{t}_0+T}_{\bar{t}_0}  | \bar{ \mathbf{u}}(t)|_{\ell_2}^2  dt \leq\frac{c_s}{2} \|\bar{y}_0\|^2_H ,
\end{equation*}
for a positive constant $c_s>0$.
P2 and P3 follow with similar arguments as in the proof of Theorem \ref{Theo2} using the inequalities \eqref{e31} and \eqref{e32},  respectively.
\end{proof}

\begin{remark}
Condition \eqref{e34} might be considered restrictive from the stabilizability point of view. However, while the uncontrolled system is stable,  the exponential stability is not clear due to the lack of integrability (see \eqref{e38}), since $\hat \nu$ is not uniformly bounded away from $0$.   Further,  since $\nu_{\min}(\omega)$ can be really small and arbitrarily close to zero for some realization of $\omega \in \Omega$,  the stability can be quite slow.  However,  using Algorithm \ref{RRHA_l} we are able to stabilize the system exponentially independent of all the perturbations of the dynamics caused by all possible realizations of the random variable $\nu$.
\end{remark}

\subsection{Failure probability}
Similarly to Section \ref{Sec5},  under Assumption \ref{assump2},  we derive an upper bound for the probability,  where the stabilizability of receding horizon framework for both problems \eqref{e50} and \eqref{e51} is not guaranteed.   Due to \eqref{e39} in Remark \ref{Rem1},  the stabilizability may fail if
\begin{equation*}
\nu_{\min}(\omega) \beta_{ \bar{N}}    \leq   3\|a\|_{ L^{\infty}((0,\infty)\times D ; \mathbb{R})},
\end{equation*}
holds.  Thus,   setting  $\Gamma(\omega):=  \sum^{\infty}_{j= 1}\abs{z_j (\omega)}\| \psi_j\|_{L^{\infty}(D)}$,  we can write for the diffusion defined in Example \ref{Examp3} that
\begin{equation*}
\begin{split}
\mathbb{P}& \left[\nu_{\min}(\omega) \beta_{ \bar{N}}    \leq   3\|a\|_{ L^{\infty}((0,\infty)\times D ; \mathbb{R})}\right]   = \mathbb{P}\left[\nu_{\min}(\omega) \leq  3\|a\|_{ L^{\infty}((0,\infty)\times D ; \mathbb{R})}\beta^{-1}_{ \bar{N}} \right]   \\ & = \mathbb{P}\left[ - \Gamma(\omega) \leq \log \left(3\|a\|_{ L^{\infty}((0,\infty)\times D ; \mathbb{R})}\beta^{-1}_{ \bar{N}} \right) \right]  = \mathbb{P}\left[ \Gamma(\omega)  \geq \log \left(\frac{\beta_{ \bar{N}}}{3\|a\|_{ L^{\infty}((0,\infty)\times D ; \mathbb{R})}}\right)    \right].
\end{split}
\end{equation*}
Further,  for every $\kappa_0>0$,  we can write
\begin{equation}
\label{e40}
\mathbb{P}  \left[ \Gamma(\omega) \geq \log \left(\frac{\beta_{ \bar{N}}}{3\|a\|_{ L^{\infty}((0,\infty)\times D ; \mathbb{R})}}\right)    \right] \leq  \mathbb{E}\left[ e^{\kappa_0\Gamma^2}   \right] e^{ -\kappa_0\log \left(\frac{\beta_{ \bar{N}}}{3\|a\|_{ L^{\infty}((0,\infty)\times D ; \mathbb{R})}} \right)^2}.
\end{equation}
where we have used the Markov inequality in the last step and we have assumed that
$\beta_{ \bar{N}} \geq 3\|a\|_{ L^{\infty}((0,\infty)\times D ; \mathbb{R})}$.
Note that due to Fernique's theorem \cite[Theorem 2.7]{MR3236753},  there exists $\kappa_0 >0$  such that the expectation of the double exponential in the left-hand side of \eqref{e40} is bounded.  Then, we can conclude that
\begin{equation}
\label{e49}
e^{-\kappa_0  \log\left(\frac{\beta_{ \bar{N}}}{3\|a\|_{ L^{\infty}((0,\infty)\times D ; \mathbb{R})}} \right)^2} = \mathcal{O}(\beta_{ \bar{N}}^{-p}(3\|a\|_{ L^{\infty}((0,\infty)\times D ; \mathbb{R})})^p )
\end{equation}
for every $p\in [1,\infty)$.  This follows from the observation that for every $\kappa>0$ and $p, x \in [1,\infty)$ we can write
\begin{equation*}
\begin{aligned}
    e^{\kappa \log(x)^2}
    \geq C_p x^p
    &\Leftrightarrow
    \kappa \log(x)^2 \geq
    \log(C_p) + p\log(x)
    &\Leftrightarrow
    \kappa \log(x)^2- p\log(x) \geq \log(C_p)
\end{aligned}
\end{equation*}
for a constant  $C_p >0$ depending only on $p$.   Further,  the function $f(y) = \kappa y^2 - py$ is lower bounded by $-p/(4\kappa)$,  that  is,  $ \min_{y\in \mathbb{R}} \kappa y^2 - py = -p^2/(4\kappa)$. Then,  it follows that for every $p, x\in [1,\infty)$,   we have $ e^{\kappa \log(x)^2}
    \geq e^{-\frac{p^2}{4\kappa}} x^p. $ Hence,  the existence  of  $C_p$ is justified by setting $C_p:=e^{-\frac{p^2}{4\kappa}}$.  Finally,  using \eqref{e12} and  \eqref{e49},  we conclude that
\begin{equation*}
\mathbb{P} \left[\nu_{\min}(\omega) \beta_{ \bar{N}}    \leq   3\|a\|_{ L^{\infty}((0,\infty)\times D ; \mathbb{R})}\right] =\mathcal{O}( \bar{N}^{-2p}( c^{-1}_{\beta}3\|a\|_{ L^{\infty}((0,\infty)\times D ; \mathbb{R})})^p )
\end{equation*}
for every $p\in [1,\infty)$.

\bibliographystyle{abbrv}
\bibliography{RHC-SPDE5}

\end{document}